\documentclass[reqno,12pt]{amsart}

\usepackage{defs_KLM}

\begin{document}

\title[Compound multivariate Hawkes processes]{Compound Multivariate Hawkes Processes:\\ Large Deviations and Rare Event Simulation}

\author{Raviar~S. Karim, Roger J.~A. Laeven, and Michel Mandjes}

\begin{abstract}
In this paper, we establish a large deviations principle for a multivariate compound process induced by a multivariate Hawkes process with random marks.
Our proof hinges on showing essential smoothness of the limiting cumulant of the multivariate compound process, resolving the inherent complication that this cumulant is implicitly characterized through a fixed-point representation.
We employ the large deviations principle to derive logarithmic asymptotic results on the marginal ruin probabilities of the associated multivariate risk process.
We also show how to conduct rare event simulation in this multivariate setting using importance sampling and prove the asymptotic efficiency of our importance sampling based estimators.
The paper is concluded with a systematic assessment of the performance of our rare event simulation procedure.
\vb

\noindent
{\sc Keywords.} Hawkes processes $\circ$ compound processes $\circ$ mutual excitation $\circ$ large deviation $\circ$ ruin probability $\circ$ rare event simulation.

\vb

\noindent
{\sc MSC 2010 subject classifications. \textit{Primary}: 60G55, 60F10; \textit{Secondary}: 60E10, 62E20, 91G05.}

\vb

\noindent
{\sc Acknowledgements and Affiliations.} 
We are very grateful to Yacine A\"it-Sahalia, Frank den Hollander,
Lex Schrijver and seminar participants at Kyushu University, 
the University of Vienna, 
the University of Amsterdam, 
the online Stochastic Networks, Applied Probability, and Performance (SNAPP) seminar 
and the Conference on Stochastic Models VII at the Mathematical Research and Conference Center
in B\c edlewo
for helpful comments.
RK and RL are with the Dept.~of Quantitative Economics, University of Amsterdam.
RL is also with E{\sc urandom}, Eindhoven University of Technology, and with C{\sc ent}ER, Tilburg University;
his research was funded in part by the Netherlands Organization for Scientific Research under grant NWO VICI 2019/20.
MM is with the Korteweg-de Vries Institute for Mathematics, University of Amsterdam. 
MM is also with E{\sc urandom}, Eindhoven University of Technology, 
and with the Amsterdam Business School, University of Amsterdam; 
his research was funded in part by the Netherlands Organization for Scientific Research under the Gravitation project N{\sc etworks}, grant 024.002.003.
Email addresses: \tt{R.S.Karim@uva.nl}, \tt{R.J.A.Laeven@uva.nl}, \tt{M.R.H.Mandjes@uva.nl}.

\vb

\noindent {\it Date}: {\today}.

\end{abstract}

\maketitle

\section{Introduction}

Mutually exciting processes, or multivariate Hawkes processes (\cite{H71,HO74}), constitute an important class of point processes, particularly suitable to describe stochastic dependences among occurrences of events across time and space.
Due to their built-in feedback mechanism, they are a natural contender to model contagious phenomena where clusters of events occur in both the temporal and spatial dimensions.
Over the past decades, Hawkes processes have been increasingly applied across a broad variety of fields,
such as finance (\cite{ACL15,ALP14,BDHM13,H18}), social interaction (\cite{CS08}), neuroscience (\cite{RS10}), seismology (\cite{O88,ILMY22}), and many others.

The key property of a Hawkes process is that it exhibits ‘self-exciting’ behavior: informally, any event instantaneously increases the likelihood of, hence potentially triggers, additional future events. 
A crucial element in its definition is the so-called {\it decay function} that quantifies how quickly the effect of an initial event on future events vanishes.
Choosing this function to be exponential renders the model Markovian, which facilitates the explicit evaluation of various relevant risk and performance metrics (e.g.,\ transient and stationary moments). 
In practical applications, however, it can be more natural to allow for other, i.e., not necessarily exponential, decay functions admitting non-Markovian, and long-memory, properties but making the analysis substantially more challenging; 
see, e.g., \cite{F19,L20}
who indicate the relevance of non-Markovian models in describing contagious phenomena.

In applied probability and mathematical risk theory, Hawkes processes have been used to model the claim arrival process, and, likewise, {\it compound}\, Hawkes processes to model the associated cumulative claim process that an insurance firm is facing; see the related literature discussed in detail below. 
In collective risk theory, {\it multivariate}\, Hawkes processes provide an appealing candidate for modeling, for example, the claim arrival process associated with technological risks; see, e.g., \cite{BGIPW17} who apply the contagion model of \cite{ACL15} to model cyber attacks and also \cite{BBH21}.
Bearing in mind that ruin and exceedance probabilities ought to be kept small, a primary research goal concerns their analysis in the asymptotic regime in which the initial reserve level of the insurer or the time horizon of aggregation grows large, such that the event of interest becomes increasingly rare.
This explains the interest in deriving large deviations principles for (compound) Hawkes processes, providing a formal tool to assess their rare event behavior, and facilitating in particular the identification of the asymptotics of ruin and exceedance probabilities.
At the same time, it is noted, however, that large deviations results usually yield rough, logarithmic asymptotics only, in that they focus on identifying the associated decay rate. 
To remedy this, one could attempt to develop `large deviations informed' simulation techniques by which rare events can be evaluated fast and accurately.
This is particularly useful when the probabilities of rare events are too small to be estimated with reasonable accuracy using regular Monte Carlo simulations.

Whereas large deviations for {\it univariate} Hawkes processes are well understood, their {\it multivariate} counterpart is to a large extent unexplored.
In this context, we mention \cite{ZBGG15}, which considers a broad class of multivariate affine processes of Markovian type, covering the special case of the multivariate Hawkes process with an exponential decay function; see also \cite{GZ19} for refinements. 
In addition, a moderate deviations result has been derived in \cite{Y18}.
To the best of our knowledge, however, large deviations principles for multivariate compound Hawkes processes allowing for general decay functions have not been established, and in addition, no rare event simulation techniques have been developed in this setting.
These are the main subjects of this paper.

In the univariate case, large deviations results for compound Hawkes processes with general decay function have been derived in \cite{ST10} building on \cite{BT07}. 
The underlying argumentation relies on the {\it cluster representation} of the driving Hawkes process, as developed in the seminal work \cite{HO74}, from which it is concluded that the cluster size follows a so-called {\it Borel} distribution.
A crucial element in proving the large deviations principle lies in showing that the limiting cumulant of the random object under study is {\it steep}, entailing that its derivative grows to infinity when approaching the boundary of its domain, 
such that the G\"artner-Ellis Theorem can be invoked.
In the univariate (compound) Hawkes case considered in \cite{ST10}, steepness could be established by using the explicit expression for the cluster size distribution.
When studying large deviations for {\it multivariate} (compound) Hawkes processes, however, a main technical difficulty that arises is that the cluster size distribution of the process is not known in closed form.
In the multivariate setting, all one has is a vector-valued fixed-point representation for the limiting cumulant of the (multivariate) cluster size, as was derived in \cite{KLM21}.
Importantly, this relation does not allow the closed-form identification of the limiting cumulant (let alone that one can find the distribution itself), entailing that one cannot explicitly characterize the boundary of its domain.

In light of the gaps in the literature described above, the contributions of this paper are the following.
\begin{itemize}
\item[$\circ$]
First, we establish a large deviations principle for multivariate compound Hawkes processes allowing for general decay functions.
We also allow the Hawkes process to be {\it marked}, such that the intensity process experiences jumps of random size constituting another potential source of rare, atypical behavior.
We have succeeded in establishing steepness based on an implicit fixed-point representation for the limiting cumulant of the joint cluster size distribution.
Specifically, without having an explicit expression for the limiting cumulant, and without having an explicit characterization of its domain, we prove that the derivative of the limiting cumulant grows to infinity when approaching the boundary of its domain.
This steepness property facilitates the use of the Gärtner-Ellis Theorem, so as to establish the desired large deviations principle.
The mathematical details of the required multivariate analysis are involved.

\item[$\circ$]
Second, we characterize the asymptotic behavior of the ruin probability for the marginal ruin processes in the regime that the initial reserve level grows large.
We prove that this ruin probability decays essentially exponentially, with the corresponding decay rate being equal to the unique zero of the limiting cumulant pertaining to that marginal. 
The proof of the lower bound on the decay rate is a direct application of our large deviations principle for multivariate compound Hawkes processes. 
The corresponding upper bound is established by a time-discretization argument, a union bound, and by showing that one can neglect the contributions to the resulting sum due to small and large time scales, in combination with frequent use of the well-known Chernoff bound.

\item[$\circ$]
Third, we develop an importance sampling algorithm for estimating rare event probabilities in our multivariate setting.
More precisely, we first derive the parameters of the exponentially twisted multivariate Hawkes process.
The identification of the exponential change of measure is non-trivial, as some of the relevant functions pertaining to the model under the original measure are only known as solutions of 
a vector-valued fixed-point representation. 
The twisted marginal ruin process has positive drift, yielding ruin with probability one under the new measure, but with the likelihood ratio being bounded by a function that decays exponentially in the initial reserve, thus leading to a considerable speedup in importance sampling relative to regular Monte Carlo simulation.
We prove that this estimator is, in fact, asymptotically efficient in the sense of Siegmund (\cite{S76}), in passing also establishing a Lundberg-type upper bound on the ruin probability.
In addition, we devise an asymptotically efficient importance sampling algorithm for estimating the probability of the multivariate compound Hawkes process (at a given point in time, that is) attaining a rare large value. 
The attainable speedup, relative to regular Monte Carlo simulation, is quantified through a series of simulation experiments.
\end{itemize}

Without attempting to provide an exhaustive overview, we now review a few important related papers, all of which focus on the {\it univariate} setting. 
We already mentioned \cite{ST10}, which analyzes the asymptotic behavior of ruin probabilities, under the assumption of light-tailed claims, drawing upon earlier large deviations results derived in \cite{BT07} for more general Poisson cluster processes.
Furthermore, in \cite{ST10}, 
an importance sampling based algorithm is proposed that is capable of efficiently generating estimates of the rare event probabilities of interest.
In \cite{KZ15}, the limiting cumulant of the cluster size distribution is implicitly characterized for the setting with random marks using a fixed-point argument, while proving a large deviations principle using the
G\"artner-Ellis theorem for the upper bound and an exponential tilting method for the lower bound.
Where the contributions above focus primarily on the case of light-tailed claims, subexponentially distributed claims are studied in \cite{Z13}, in the context of a non-stationary version of the Hawkes process.
For (non-compound) Hawkes processes (i.e., not involving claims), `precise' large deviations results, providing asymptotics beyond the leading order term, are obtained in \cite{GZ21}.
The setting of a large initial intensity is studied in \cite{GZ18a} and \cite{GZ18b}. 
For the more general class of non-linear Hawkes processes, \cite{Z14} proves the process-level large deviations, and \cite{Z15} derives large deviations in the Markovian setting.

The rest of this paper is organized as follows. 
In Section~\ref{sec:ModelBranching}, we introduce the relevant processes and discuss some basic properties that are used throughout the paper.
Section~\ref{sec:TransformAnalysis} derives results on the transform of the joint cluster size distribution,
and provides an implicit characterization of the domain of the limiting cumulant.
Section~\ref{sec:LargeDevs} establishes the large deviations principle for the multivariate compound Hawkes process with general decay function and random marks.
In Section~\ref{sec:RuinProbs}, we consider the associated multivariate risk process, with the objective to characterize the decay rate of the marginal ruin probability.
In Section~\ref{sec:simulation}, we exploit the large deviations principle to develop an importance sampling algorithm to efficiently estimate rare event probabilities.
Section~\ref{sec:Numerics} numerically demonstrates the performance of our importance sampling based estimators.
Concluding remarks are in Section~\ref{sec:Con}. 
Some auxiliary proofs are relegated to the Appendix.

\section{Multivariate Compound Hawkes Processes}\label{sec:ModelBranching}

In this section, we first provide the definitions of multivariate Hawkes and compound Hawkes processes, and next introduce some objects and discuss some properties that are relevant in the context of this paper.
Throughout, we use the boldface notation ${\bm x} =(x_1,\ldots,x_d)^\top$ to denote a $d$-dimensional vector, for a given dimension $d\in{\mathbb N}$.
Inequalities between vectors are understood componentwise, e.g., $\bm{x} > \bm{y}$ means $x_i > y_i$ for all $i=1,\dots,d$.

Consider a $d$-dimensional c\`adl\`ag counting process ${\bm N}(\cdot)\equiv ({\bm N}(t))_{t\in{\mathbb R_+}}$,
where each increment $N_i(t) - N_i(s)$ records the number of points in component $i\in[d]:=\{1,\dots,d\}$
in the time interval $(s,t]$, with $s<t$.
We label the points by considering, for each $j\in[d]$, a sequence of a.s.\ increasing positive random variables $\bm{T}_j = \{T_{j,r}\}_{r\in\nn} = \{T_{j,1},T_{j,2},\dots\}$ representing event times.
We associate to this sequence the one-dimensional counting process $N_j(\cdot)$ by setting \[N_j(t) := N_{\bm{T}_j}(0,t] =\sum_{r=1}^\infty \bm{1}_{\{T_{j,r}\leq t\}}.\]
The process $\bm{N}(\cdot) = (N_1(\cdot),\dots,N_d(\cdot))^\top$ is then the $d$-dimensional counting process associated with the sequences of event times in all components, $\bm{T}_1,\dots,\bm{T}_d$, compactly denoted by $\bm{N}(t) = \bm{N}_{\bm{T}}(0,t]$.
Throughout, the points will be referred to as \textit{events} and
the terms point process and counting process are used interchangeably for $\bm{N}(\cdot)$.
We assume that the point process starts empty, i.e., $\bm{N}(0) = \bm{0} = (0,\dots,0)^\top$.

In the original work \cite{H71}, the Hawkes process is defined by relying on the concept of the conditional intensity function.
An alternative, equivalent definition, known as the cluster process representation, can be given by representing the Hawkes process as a Poisson cluster process; 
it was first described in \cite{HO74} in the setting of the conventional univariate Hawkes process, see also \cite[Example 6.3(c)]{DVJ07} and \cite[Ch.\ IV]{L09}.
The cluster process representation distinguishes between two types of events: 
first, there are \textit{immigrant} events generated according to a homogeneous Poisson process with a given rate; 
and second, there are \textit{offspring} events generated by an inhomogeneous Poisson process with rates that account for self-excitation and, in the multivariate context also, cross-excitation.
In the following, we introduce the relevant terminology and provide a formal definition of the process.

For $j\in[d]$, we consider {\it base rates} $\overline{\lambda}_j \geqslant 0$, with at least one of the base rates being strictly positive.
For each combination $i,j\in[d]$, we let the {\it decay function} $g_{ij}(\cdot):\mathbb{R}_{+}\rightarrow\mathbb{R}_{+}$ be non-negative, non-increasing, and integrable. 
Also, for $j\in[d]$, we define the {\it random marks} through the generic non-negative, non-degenerate random vector $\bm{B}_{j} = (B_{1j},\dots,B_{dj})$, asserting that the sequence of random marks $\{\bm{B}_{j,r}\}_{r\in\nn}$ consists of i.i.d.\ random vectors that are distributed as $\bm{B}_{j}$.
We allow the random variables $B_{ij,r}$ to be dependent for fixed $j$ and $r$.
Finally, we let $K_{ij}(\cdot)$ be an inhomogeneous Poisson process with intensity $B_{ij,r}\,g_{ij}(\cdot)$, given the value of $B_{ij,r}$.
With these elements in place, the following definition describes the cluster process representation for a multivariate Hawkes process.

\begin{definition}[{multivariate Hawkes process}]
\label{def: hawkes cluster}
Define a $d$-dimensional point process $\bm{N}(\cdot)$ componentwise by $N_j(t) = N_{\bm{T}_j}(0,t]$ for $j\in[d]$ and $t>0$, where the sequences of event times $\bm{T}_1,\dots,\bm{T}_d$ are generated as follows:
\begin{enumerate}
    \item[{\em(i)}] First, for each $j\in[d]$, let there be a sequence of immigrant event times $\{T_{j,r}^{(0)}\}_{r\in\nn}$ on the interval $(0,\infty)$ generated by a homogeneous Poisson process $I_j(\cdot)$ with rate $\overline{\lambda}_j$.
    \item[{\em(ii)}] Second, let each immigrant event \textit{independently} generate a $d$-dimensional cluster $\bm{C}_j \equiv \bm{C}_{T_{j,r}^{(0)}}$, consisting of event times associated with generations of events:
    \begin{enumerate}
        \item[{\em(a)}] The immigrant with event time $T_{j,r}^{(0)}$ is labeled to be of generation $0$ and into each component $m\in[d]$, it generates a sequence of first-generation event times $\{T_{m,r}^{(1)}\}_{r\in\nn}$ on the interval $(T_{j,r}^{(0)},\infty)$, according to $K_{mj}(\cdot - T_{j,r}^{(0)})$ with $B_{mj,r}$ the random mark associated to $T_{j,r}^{(0)}$.
        \item[{\em(b)}] Iterating {\em (a)} above, with $T_{m,r}^{(n-1)}$ designating the $r$-th event time of generation $n-1$ in component $m\in[d]$, yields generation $n$ event times $\{T_{l,r}^{(n)}\}_{r\in\nn}$ in component $l\in[d]$ on the interval $(T_{m,r}^{(n-1)}, \infty)$, generated according to $K_{lm}(\cdot - T_{m,r}^{(n-1)})$.
    \end{enumerate}
\end{enumerate}
Upon taking the union over all generations, we obtain, for each component $j\in[d]$,
\begin{align*}
    \bm{T}_j = \{T_{j,r}\}_{r\in\nn} = \bigcup_{n=0}^\infty \{T_{j,r}^{(n)}\}_{r\in\nn}.
\end{align*}
The process $\bm{N}(\cdot)$ defined above for $t>0$ and with $\bm{N}(0)=\bm{0}$ constitutes a multivariate Hawkes process.
\end{definition}

To ensure that the clusters described in part~(ii) of Definition~\ref{def: hawkes cluster} are a.s.\ finite, we assume that a \textit{stability condition} applies throughout this paper.
It is shown in \cite{H71} that this stability condition guarantees non-explosiveness of $\bm{N}(\cdot)$, see also \cite[Example 8.3(c)]{DVJ07}.

\begin{assumption} \label{ass: stability condition}
Assume that the matrix $\bm{H} := (h_{mj})_{m,j\in[d]}$ with elements
\begin{align}
\label{eq: def mu multivariate}
    h_{mj} := \ee[B_{mj}] c_{mj},
\end{align}
with $c_{mj} = \int_0^\infty g_{mj}(v)\,\ddiff v$, has spectral radius strictly smaller than $1$.
\end{assumption}

We next define the multivariate compound Hawkes process as follows.
Let ${d^\star}\in\nn$ be fixed and note that we allow $d \neq {d^\star}$.
For each $j\in[d]$, let $\{\bm{U}_{j,r}\}_{r\in\nn} = \{(U_{1j,r},\dots,U_{{d^\star}j,r})^\top\}_{r\in\nn}$ be a sequence of non-negative, non-degenerate i.i.d.\ random vectors of length ${d^\star}$.
We allow the random variables $U_{ij,r}$ to be dependent for fixed $j$ and $r$. 

\begin{definition}[multivariate compound Hawkes process]
\label{def: compound Hawkes}
Define $\bm{Z}(\cdot) := (Z_1(\cdot), \ldots, Z_{{d^\star}}(\cdot))^{\top}$ for each component $Z_i(\cdot)$ with $i\in[{d^\star}]$ by
\begin{align}
\label{eq: compound hawkes entry}
    Z_i(t) := \sum_{j=1}^d \sum_{r=1}^{N_j(t)} U_{ij,r},\quad t>0,
\end{align}
where $\bm{U}_{j,r}= (U_{1j,r},\dots,U_{{d^\star}j,r})^\top$ is drawn independently for every event in $N_j(t)$, with $j\in[d]$. 
The process $\bm{Z}(\cdot)$ defined above for any $t>0$ constitutes a multivariate compound Hawkes process.
\end{definition}

If we define the random matrix $\bm{U}\in \rr^{{d^\star} \times d}_{+}$ as
\begin{align}
    \bm{U} \equiv \begin{bmatrix}
    \bm{U}_1,\dots,\bm{U}_d
    \end{bmatrix}
    := \begin{bmatrix}
    U_{11} & U_{12} & \dots & U_{1d} \\
    U_{21} & U_{22} & \dots & U_{2d} \\
    \vdots & \vdots & \ddots & \vdots \\
    U_{{d^\star}1} & U_{{d^\star}2} & \cdots & U_{{d^\star}d}
    \end{bmatrix},
\end{align}
we can represent Eqn.~\eqref{eq: compound hawkes entry} in vector-matrix form by
\begin{align}
    \bm{Z}(t) = \bm{N}(t) \star\bm{U} := \sum_{j=1}^d \sum_{r=1}^{N_j(t)} \bm{U}_{j,r},
\end{align}
where $\star$ denotes the compound sum operation.

We proceed with a brief discussion of the dimensionality of the objects appearing in Definition~\ref{def: compound Hawkes}.
The Hawkes process $\bm{N}(\cdot)$ is of dimension $d$ and the random vectors $\bm{U}_j$ are of dimension ${d^\star}$, which results in the compound Hawkes process $\bm{Z}(\cdot)$ being also of dimension ${d^\star}$.
This reflects that the random variables of the type $U_{ij}$, with $i\in[{d^\star}]$ and $j\in[d]$, can be interpreted as the effect that an event in the $j$-th component of the Hawkes process $\bm{N}(\cdot)$ has on the $i$-the component of the compound Hawkes process $\bm{Z}(\cdot)$.
As stated before, we allow $d\neq {d^\star}$.
Intuitively, for instance, in the context of insurance, this means that the number of risk drivers may be larger $(d> {d^\star})$ or smaller $(d < {d^\star})$ than the number of insurance product categories.

We now introduce some objects related to the cluster process representation that are relevant for later analysis.
Recall that for each immigrant in component $j\in[d]$, the $d$-dimensional cluster $\bm{C}_j$ from Definition~\ref{def: hawkes cluster} contains the sequences of event times in each component that have the immigrant with event time $T_{j,r}^{(0)}$ as oldest ancestor.
We associate to $\bm{C}_j$ the $d$-dimensional \textit{cluster point process} $\bm{S}_j(\cdot)$, by setting 
\begin{align}
    \bm{S}_j(u) := \bm{S}_{\bm{C}_j}(0,u],
\end{align}
such that it counts the number of events of $\bm{C}_j$ on the interval $(0,u]$, where $u= t-T_{j,r}^{(0)} >0$ is the remaining time after the arrival of the immigrant event.
Concretely, we have
\begin{align}
    \bm{S}_j(u) :=
    \begin{bmatrix}
    S_{1\leftarrow j}(u) \\
    \vdots \\
    S_{d\leftarrow j}(u)
    \end{bmatrix},
\end{align}
where each entry $S_{i\leftarrow j}(u)$ records the number of events generated into component $i\in[d]$ in the cluster $\bm{C}_j$, with as oldest ancestor the immigrant event in component $j$ that generated the cluster.
To avoid double counts, we let the immigrant itself be included in the cluster (only) when $i = j$.

If we let $u$ tend to infinity, the entries of the random vector $\bm{S}_j(u)$ ultimately count the \textit{total} number of events within the cluster $\bm{C}_j$ generated into each component $i\in[d]$.
Observe that $u\mapsto \bm{S}_j(u)$ is increasing componentwise and $\sup_{u\in\rr_+}\norm{\bm{S}_j(u)}_{\rr^d} < \infty$ with probability 1 due to Assumption~\ref{ass: stability condition}.
Hence, we can define a random vector that counts the total number of events in all components, or simply {\it cluster size}, by setting $\bm{S}_j := \lim_{u\to\infty} \bm{S}_j(u)$, where convergence is understood in the a.s.\ sense. 

One can interpret these clusters in terms of $d$-type Galton-Watson processes, where the total progeny equals the sum of all generations of offspring that descend from one individual (\cite{J75}).
Suppose the Galton-Watson process starts with an individual of type $j\in[d]$, and let $\bm{S}_j^{(k)}$ denote the $k$-th generation of descendants.
Then one can write
\begin{align}
\label{eq: cluster equals sum of generations}
    \bm{S}_j = \sum_{k=0}^\infty \bm{S}_j^{(k)},
\end{align}
where $\bm{S}_j^{(0)} = \bm{e}_j$, the unit vector (i.e., with $j$-th entry equal to $1$, and other entries equal to 0).
In \cite{J75}, the total progeny, i.e. cluster size, of such a process is analyzed in the one-dimensional setting and shown to have a so-called Borel distribution.
For higher-dimensional Hawkes processes, by using~\cite[Theorem 1.2]{CL14}, it is, in principle, also possible to derive a representation of the multivariate cluster size distribution.
However, the resulting expression is neither explicit nor workable for the goal at hand due to (highly) convoluted sums that arise in the derivation.
More specifically, the multiplicity of the different possible sample paths to generate a certain number of events in each component yields a complex combinatorial problem. 

We conclude this section by stating two convergence results that will be needed later in this paper.
Under the stability condition, we have that the Hawkes process $\bm{N}(\cdot)$ satisfies the following strong law of large numbers, as shown in \cite{BDHM13}:
as $t\to\infty$, we have
\begin{align}
\label{eq: LLN for hawkes process N}
    \frac{\bm{N}(t)}{t} \to 
    (\bm{I} - \bm{H})^{-1}\overline{\bm{\lambda}}, 
\end{align}
a.s., where $\bm{\overline{\lambda}} = (\overline{\lambda}_1,\dots,\overline{\lambda}_d)^\top$.
This result naturally extends to the corresponding compound Hawkes process $\bm{Z}(\cdot)$
(see \cite[Theorem 1]{GRS20}):
as $t\to\infty$, a.s., 
\begin{align}
\label{eq: LLN for compound hawkes Z}
    \frac{\bm{Z}(t)}{t} \to\ee[\bm{U}] (\bm{I} - \bm{H})^{-1}\overline{\bm{\lambda}}. 
\end{align}



\section{Transform Analysis}\label{sec:TransformAnalysis}
In this section, we discuss probability generating functions and moment generating functions pertaining to the processes introduced in the previous section, viz.\ the multivariate Hawkes and compound Hawkes processes.
These functions will play a pivotal role when deriving large deviations results later in this paper. 

It is directly seen from the definition of $\bm{Z}(t)$ that, for fixed $t>0$, its moment generating function satisfies
the following composite expression in terms of the probability generating function of $\bm{N}(t)$:
\begin{align}
\label{eq: compound hawkes characterization}
    m_{\bm{Z}(t)}(\bm{\theta}) 
    \equiv \ee\big[e^{\bm{\theta}^\top \bm{Z}(t)}\big]
    = \ee\Big[ \prod_{l=1}^d \big(m_{\bm{U}_l}(\bm{\theta})\big)^{N_l(t)} \Big],
\end{align}
where
\begin{align*}
    m_{\bm{U}_l}(\bm{\theta}) \equiv \ee[e^{\bm{\theta}^\top \bm{U}_l}] = \ee\Big[ \prod_{i=1}^{{d^\star}} e^{\theta_i U_{il}}\Big].
\end{align*}
For now, we assume $\bm{\theta}\in\rr^{{d^\star}}$ is chosen such that \eqref{eq: compound hawkes characterization} exists---we will further discuss the domain of convergence below.
We are interested in the limiting cumulant of $\bm{Z}(t)$ as $t\to\infty$, that is, we wish to analyze 
\begin{align} \label{eq: limit cumulant of Z}
    \lim_{t\to\infty} \frac{1}{t} \log m_{\bm{Z}(t)}(\bm{\theta}).
\end{align}

To derive an expression for~\eqref{eq: limit cumulant of Z}, we use a characterization of the probability generating function of $\bm{N}(t)$ in terms of the cluster point processes $\bm{S}_j(u)$, obtained in~\cite[Theorem 1]{KLM21}:
\begin{align}
\label{eq: pgf hawkes characterization}
    \ee\Big[ \prod_{l=1}^d z_l^{N_l(t)}\Big]
    &= \prod_{j=1}^d \exp\Big(\overline{\lambda}_j\int_0^t \big(\ee\Big[\prod_{l=1}^d z_l^{S_{l\leftarrow j}(u)}\Big] - 1\big)  \,\mathrm{d}u\Big),
\end{align}
where, for each $j\in[d]$, the probability generating function of $\bm{S}_j(u)$ appearing on the right-hand side of~\eqref{eq: pgf hawkes characterization} satisfies the fixed-point representation
\begin{align}\label{eq: pgf clusters fixed point time dependent}
    f_j(\bm{z},u)
    &:=\ee\Big[\prod_{l=1}^d z_l^{S_{l\leftarrow j}(u)}\Big]
    =  z_j \ee\Big[ \exp\Big(\sum_{m=1}^d B_{mj} \int_0^u g_{mj}(v)\Big(f_m(\bm{z},u-v) -1 \Big)\,\mathrm{d}v\Big)\Big];
\end{align}
see~\cite[Theorem 2]{KLM21}.

In order to exploit this characterization to establish our large deviations result in the multivariate compound Hawkes setting, we need to analyze the domain of $\bm{z} = (z_1,\dots,z_d)^\top$ for which Eqns.~\eqref{eq: pgf hawkes characterization} and \eqref{eq: pgf clusters fixed point time dependent} are valid, i.e., where the probability generating functions of $\bm{S}_j(u)$ exist.
More precisely, since we focus on the regime $t\to\infty$ in~\eqref{eq: limit cumulant of Z}, we need to consider the probability generating function of the total cluster size $\bm{S}_j$ instead of $\bm{S}_j(u)$, which in the sequel is denoted by
\begin{align}
    f_j(\bm{z}) := \lim_{u\to\infty} f_j(\bm{z},u) = \ee\Big[\prod_{l=1}^d z_l^{S_{l\leftarrow j}}\Big],
\end{align}
and uncover its domain.
Let $\bm{f}:\rr_+^d \to \overline{\rr}^d$ be given by $\bm{f}(\bm{z}) = (f_1(\bm{z}),\dots,f_d(\bm{z}))^\top$ and denote its effective domain by ${\mathscr D}_f := \{ \bm{z} \in \rr_+^d : \norm{\bm{f}(\bm{z})}_{\rr^d} < \infty\}$.
For the set ${\mathscr D}_f$, denote the interior by ${\mathscr D}_f^\circ$ and the boundary by $\partial{\mathscr D}_f$.

Observe that the right-hand side of Eqn.~\eqref{eq: pgf clusters fixed point time dependent} is expressed in terms of the moment generating function of the random vector $\bm{B}_j$.
We assume the following to hold throughout the paper.

\begin{assumption} \label{ass: mgfs B exist}
Assume that for some $\bm{\vartheta} \in \rr_+^{d}$,
\begin{align}
    m_{\bm{B}_j}(\bm{\vartheta})
    =\ee\Big[\exp\Big(\sum_{m=1}^d B_{mj}\vartheta_m\Big)\Big]
    < \infty,
\end{align}
for all $j\in[d]$.
\end{assumption}

The following result gives an implicit characterization of $\bm{f}(\cdot)$ and its domain ${\mathscr D}_f$ in terms of a fixed-point representation.

\begin{proposition} \label{prop: characterize f(z) and D_f explicit}
The vector-valued function $\bm{f}(\bm{z})$ is the unique increasing function that satisfies
\begin{align} \label{eq: fixed point equation pgf of S_j (steady state)}
    f_j(\bm{z})  
    = z_j \,\ee\Big[\exp\Big(\sum_{m=1}^d B_{mj}c_{mj}(f_m(\bm{z}) - 1)\Big)\Big],
\end{align}
for $\bm{z}\in\rr_+^d$ such that $\bm{z} \leqslant \hat{\bm{z}}_{\bm{r}} \equiv \hat{\bm{z}}$,
where, for an arbitrarily given positive vector $\bm{r} \in \rr_+^d$, $\hat{\bm{z}}=(\hat{z}_{1},\dots,\hat{z}_{d})^\top$ is given for each $j\in[d]$ by 
\begin{align} \label{eq: zstar solution multivariate}
    \hat{z}_{j} = r_j \left( \sum_{k=1}^d r_k \ee\Big[ B_{kj}c_{kj}\exp\Big(\sum_{m=1}^d B_{mj}c_{mj}(\hat{x}_m - 1)\Big) \Big] \right)^{-1},
\end{align}
and $\hat{\bm{x}} = (\hat{x}_1,\dots,\hat{x}_d)^\top$ solves
\begin{align} \label{eq: system eqn xstar multivariate}
    &x_j \left( \sum_{k=1}^d r_k \ee\Big[ B_{kj}c_{kj}\exp\Big(\sum_{m=1}^d B_{mj}c_{mj}(x_m - 1)\Big) \Big] \right) \nonumber\\ 
    &= r_j\ee\Big[ \exp\big(\sum_{m=1}^d B_{mj}c_{mj}(x_m - 1)\Big)\Big].
\end{align}
\end{proposition}

\begin{proof}
The proof consists of three parts: \textit{(i)} identifying the limit of $f_{j}(\bm{z},u)$ as $u\to\infty$; \textit{(ii)} implicit characterization of the domain ${\mathscr D}_f$; \textit{(iii)} proving uniqueness of $\bm{f}(\cdot)$. 

\textit{--- Proof of (i)}. 
We show that for $\bm{z} \in {\mathscr D}_f$, we have that $f_j(\bm{z},u) \to f_j(\bm{z})$ for all $j\in[d]$.
At this point, we do not yet know the precise domain ${\mathscr D}_f$, but we do know it is a convex subset of $\rr_+^d$ and we implicitly derive it later in the proof.

When $\bm{z} =\bm{1}$, we have $\bm{f}(\bm{1},u) \equiv \bm{1} \equiv \bm{f}(\bm{1})$ and convergence follows trivially.
Observe that when $\bm{0} \leqslant \bm{z} < \bm{1}$, $f_j(\bm{z},u)$ is decreasing in $u$ and $0\leqslant f_j(\bm{z},u) < 1$,
hence, $f_j(\bm{z},u)$ converges by the monotone convergence theorem to a finite limit as $u\rightarrow\infty$, satisfying the limit of \eqref{eq: pgf clusters fixed point time dependent}.
When $\bm{z}> \bm{1}$, $f_j(\bm{z},u)$ is increasing in $u$ and either diverges to $\infty$ or converges to a finite limit, satisfying the limit of \eqref{eq: pgf clusters fixed point time dependent}.
In the intermediate case, where for some $k,m\in[d]$ one has $z_k \leqslant 1$ and $z_m > 1$, we proceed as follows.
Recall that for each $j\in[d]$, the map $u \mapsto \bm{S}_j(u)$ is a.s. increasing in all components.
We obtain the following upper bound:
\begin{align*}
    \limsup_{u\to\infty}f_j(\bm{z},u) 
    &=\limsup_{u\to\infty} \ee\Big[\prod_{i=1}^d z_i^{S_{i\leftarrow j}(u)}\Big]  \\
    &= \limsup_{u\to\infty} \ee\Big[\prod_{k:z_k\leqslant 1} z_k^{S_{k\leftarrow j}(u)}
    \prod_{m:z_m> 1} z_m^{S_{m\leftarrow j}(u)}\Big] \\
    &\leqslant \limsup_{u\to\infty} \ee\Big[\prod_{k:z_k\leqslant 1} z_k^{S_{k\leftarrow j}(u)}
    \prod_{m:z_m> 1} z_m^{S_{m\leftarrow j}}\Big] \\
    &\overset{(*)}{=}
    \ee\Big[\prod_{i=1}^d z_i^{S_{i\leftarrow j}}\Big] = f_j(\bm{z}),
\end{align*}
if the limit is finite, where in $(*)$ we have used the monotone convergence on the product over $k$, as this product is decreasing.
Similarly, we obtain the lower bound
\begin{align*}
    \liminf_{u\to\infty}f_j(\bm{z},u) 
    &\geqslant \liminf_{u\to\infty} \ee\Big[\prod_{k:z_k\leqslant 1} z_k^{S_{k\leftarrow j}}
    \prod_{m:z_m> 1} z_m^{S_{m\leftarrow j}(u)}\Big] \\
    &\overset{(*)}{=}
    \ee\Big[\prod_{i=1}^d z_i^{S_{i\leftarrow j}}\Big] = f_j(\bm{z}),
\end{align*}
which implies that the $\liminf$ and $\limsup$ coincide, and so $f_j(\bm{z},u)\to f_j(\bm{z})$ for all $j\in[d]$.
Provided all components converge, we have convergence of the vector $\bm{f}(\bm{z},u)\to\bm{f}(\bm{z})$.
Hence, from~\eqref{eq: pgf clusters fixed point time dependent} we obtain for each $j\in[d]$ and with $\bm{z}\in {\mathscr D}_f$ that
\begin{align}
    f_j(\bm{z}) 
    = z_j \,\ee\Big[\exp\Big(\sum_{m=1}^d B_{mj}c_{mj}(f_m(\bm{z}) - 1)\Big)\Big],
\end{align}
yielding a vector-valued fixed-point representation for $\bm{f}(\bm{z})$.

\textit{--- Proof of (ii)}.
We now implicitly characterize the domain ${\mathscr D}_f$.
To that end, consider the function $\bm{G}:\rr^d \times \rr^d \to \rr^d$, where each entry $G_j:\rr^d \times \rr^d \to \rr$, with $j\in[d]$, is given by
\begin{align}
    G_j(\bm{z},\bm{x}) 
    = z_j \,\ee\Big[\exp\Big(\sum_{m=1}^d B_{mj}c_{mj}(x_m - 1)\Big)\Big] - x_j.
\end{align}
Note that $\bm{G}(\cdot)$ is continuously differentiable for all $\bm{z} \in \rr^d$ and $\bm{x}\in\rr^d$ for which the respective moment generating functions of $\bm{B}_j$ are defined.
To obtain a characterization of ${\mathscr D}_f$, we need to find the set of $\bm{z}$ on the boundary of ${\mathscr D}_f$, for which $\bm{f}(\bm{z})$ exists and satisfies~\eqref{eq: fixed point equation pgf of S_j (steady state)}.
Since Eqn.~\eqref{eq: fixed point equation pgf of S_j (steady state)} is analogous to solving $\bm{G}(\bm{z}, \bm{f}(\bm{z})) = \bm{0}$, we can find the domain of $\bm{f}(\cdot)$ by investigating the set $\bm{G}^{-1}(\bm{0}) = \{(\bm{z},\bm{x}) : \bm{G}(\bm{z},\bm{x}) = \bm{0}\}$.
Since the preimage $\bm{G}^{-1}(\bm{0})$ can be a complicated set, we resort to the preimage theorem, a variation of the implicit function theorem also known as the regular level set theorem, see e.g., \cite[Theorem 9.9]{T11}, which states two results.
First, the preimage has codimension equal to the dimension of the image, and second, the tangent space at a point of the preimage coincides with the kernel of the Jacobian at that point, provided that the Jacobian is of full-rank.

We proceed by providing a specification of the preimage theorem in our setting.
The first part of the preimage theorem states that $\bm{G}^{-1}(\bm{0})$ is a $d$-dimensional space.
Note that in the univariate ($d=1$) setting, $\bm{G}^{-1}(0)$ would be a curve embedded in $\rr\times \rr$, and the tangent space would be a line.
In our multivariate ($d>1$) setting, $\bm{G}^{-1}(\bm{0})$ is a $d$-dimensional manifold embedded in $\rr^d\times \rr^d$, and the tangent space is again $d$-dimensional.
The second part concerns tangent spaces, defined as follows: for any $(\bm{z},\bm{x}) \in \bm{G}^{-1}(\bm{0})$, the tangent space $T_{\bm{z},\bm{x}}(\bm{G}^{-1}(\bm{0}))$ consists of the set of vectors $(\bm{q},\bm{r})\in \rr_+^d\times\rr_+^d$ for which there exists a curve $\gamma \subseteq \bm{G}^{-1}(\bm{0})$ with $\gamma(0)=(\bm{z},\bm{x})$ and $\gamma'(0)=(\bm{q},\bm{r})$.
The second part of the preimage theorem then states
\begin{align} \label{eq: kernel is tangent space preimage thm}
    \Ker\big(\bm{J}_{\bm{G}}(\bm{z},\bm{x})\big) = T_{\bm{z},\bm{x}}(\bm{G}^{-1}(\bm{0})),
\end{align}
with $\bm{J}_{\bm{G}}(\bm{z},\bm{x}) \in \rr^{d\times 2d}$ denoting the full Jacobian of $\bm{G}$ evaluated at $(\bm{z},\bm{x})$.
We compute the $d\times d$-dimensional Jacobian matrices of partial derivatives of $\bm{G}$ w.r.t.~$\bm{z}$ and $\bm{x}$ separately by
\begin{align*}
    \bm{J}_{\bm{G},\bm{z}}
    := \left[
        \frac{\partial G_j}{\partial z_k}(\bm{z},\bm{x}) 
    \right]_{j,k\in[d]}
    = \left[ \bm{1}_{\{j=k\}}\ee\Big[\exp\Big(\sum_{m=1}^d B_{mj}c_{mj}(x_m - 1)\Big)\Big]
    \right]_{j,k\in[d]},
\end{align*}
and
\begin{align*}
    \bm{J}_{\bm{G},\bm{x}} 
    := \left[
        \frac{\partial G_j}{\partial x_k}(\bm{z},\bm{x}) 
    \right]_{j,k\in[d]}
    = \left[ z_j \ee\Big[B_{kj}c_{kj}\exp\Big(\sum_{m=1}^d B_{mj}c_{mj}(x_m - 1)\Big)\Big] - \bm{1}_{\{j=k\}}\right]_{j,k\in[d]},
\end{align*}
such that $\bm{J}_{\bm{G}} = ( \bm{J}_{\bm{G},\bm{z}} \, | \, \bm{J}_{\bm{G},\bm{x}})$.
To utilize Eqn.~\eqref{eq: kernel is tangent space preimage thm}, we need to look for vectors $(\bm{q},\bm{r}) \in \rr_+^d \times \rr_+^d$ such that $\bm{J}_{\bm{G}} \cdot (\bm{q},\bm{r}) = \bm{J}_{\bm{G},\bm{z}} \cdot \bm{q} + \bm{J}_{\bm{G},\bm{x}} \cdot \bm{r} =\bm{0}$, where we denote $\bm{J}_{\bm{G}}  \equiv \bm{J}_{\bm{G}}(\bm{z},\bm{x})$ for brevity.
Note that $\bm{J}_{\bm{G},\bm{z}}$ is a diagonal matrix with positive entries, such that $\bm{J}_{\bm{G},\bm{z}} \cdot \bm{q} = \bm{0}$ only if $\bm{q} = \bm{0}$; so we can focus on $\bm{r}$.
Observe that the set of points where the determinant of the Jacobian $\bm{J}_{\bm{G},\bm{x}}$ has the same sign is a connected set, due to strict convexity in each entry of the functions $G_j(\bm{z},\bm{x})$, with $j\in[d]$, and by continuity of the determinant and partial derivatives.
For a fixed vector $\bm{r}\in\rr_+^d$, we can establish systems of equations for $\bm{z}$ and $\bm{x}$ such that $\bm{J}_{\bm{G},\bm{x}} \cdot \bm{r} = \bm{0}$.
As will become clear later in the proof, convexity will play a crucial role in determining uniqueness of these solutions.

With the objective of substantiating the claim in~\eqref{eq: kernel is tangent space preimage thm}, we compute the solution to the systems of equations $\bm{G}(\bm{z},\bm{x}) = \bm{0}$ and $\bm{J}_{\bm{G},\bm{x}} \cdot \bm{r} = \bm{0}$ by using the expression for $\bm{J}_{\bm{G},\bm{x}}$ and solving for $\bm{z}$ and $\bm{x}$.
This yields the solutions $\hat{\bm{z}}= (\hat{z}_1,\dots,\hat{z}_d)$ and $\bm{\hat{x}} = (\hat{x}_1,\dots\hat{x}_d)$ given in Eqns.~\eqref{eq: zstar solution multivariate} and \eqref{eq: system eqn xstar multivariate}, with $(\hat{\bm{z}},\hat{\bm{x}}) \equiv (\hat{\bm{z}}_{\bm{r}},\hat{\bm{x}}_{\bm{r}})$ parameterized by vectors $\bm{r}\in \rr_+^d$, such that $\bm{G}(\hat{\bm{z}},\hat{\bm{x}}) = \bm{0}$ and $\bm{J}_{\bm{G},\bm{x}}(\hat{\bm{z}},\hat{\bm{x}}) \cdot \bm{r} = \bm{0}$.
Moreover, for any given $\bm{r}\in \rr_+^d$, we show that the associated pair $(\hat{\bm{z}}_{\bm{r}},\hat{\bm{x}}_{\bm{r}})$ is unique.
The condition $\bm{J}_{\bm{G},\bm{x}}(\hat{\bm{z}}_{\bm{r}},\hat{\bm{x}}_{\bm{r}})\cdot \bm{r} = \bm{0}$ stated for each row yields $\nabla_{\bm{x}} G_j(\hat{\bm{z}}_{\bm{r}},\hat{\bm{x}}_{\bm{r}}) \cdot \bm{r} = 0$ for all $j\in[d]$, with $\nabla_x G_j(\cdot)$ the $j$-th row of the Jacobian.
Note that each $G_j(\cdot)$ only depends on $z_j$ and $\bm{x}$.
Due to strict convexity of $G_j(\bm{z},\bm{x})$ in each entry, we have that the sub-level set $G_j^{-1}(\leqslant 0) :=\{(\bm{z},\bm{x}) : G_j(\bm{z},\bm{x}) \leqslant 0\}$ is a strictly convex set, and the level set $G_j^{-1}(0)$ is the boundary of $G_j^{-1}(\leqslant 0)$.
This implies that 
\begin{align*}
    \bm{G}^{-1}(\bm{0}) = \{(\bm{z},\bm{x}): G_j(\bm{z},\bm{x}) = 0, \forall j\in[d]\} = \bigcap_{j=1}^d G_j^{-1}(0),
\end{align*}
is the boundary of a strictly convex set, namely $\bm{G}^{-1}(\leqslant \bm{0})$, as the latter is the intersection of strictly convex sets.
Since $\bm{J}_{\bm{G},\bm{x}}(\hat{\bm{z}}_{\bm{r}},\hat{\bm{x}}_{\bm{r}})\cdot \bm{r} = \bm{0}$ means $\bm{r} \in T_{\hat{\bm{z}}_{\bm{r}},\hat{\bm{x}}_{\bm{r}}}(\bm{G}^{-1}(\bm{0}))$ by Eqn.~\eqref{eq: kernel is tangent space preimage thm}, and since $\bm{G}^{-1}(\bm{0})$ is the boundary of a strictly convex set, we have that $\bm{r}$ uniquely determines the point $(\hat{\bm{z}}_{\bm{r}},\hat{\bm{x}}_{\bm{r}})$. 

The next step amounts to relating what we found so far to the domain ${\mathscr D}_f$.
A given value of $\bm{z} \in \rr_+^d$ determines whether one can find $\bm{x}\in\rr_+^d$ for which $\bm{G}(\bm{z},\bm{x}) = \bm{0}$, such that $(\bm{z},\bm{x})\in \bm{G}^{-1}(\bm{0})$.
Observe that the set $R_z := \{\bm{z}\in\rr_+^d : \bm{z} = \hat{\bm{z}}_{\bm{r}}, \bm{r}\in\rr_+^d\}$ divides the positive quadrant $\rr_+^d$ into two disjoint sets.
The first set is the inner (convex) region, defined as the set of $\bm{z}\in\rr_+^d$ enclosed by the origin, the $\bm{z}$ axes and the set $R_z$, with $R_z$ included; denote this set by $\mathcal{Z}$.
The second set is the outer region, denoted by $\mathcal{Z}^c$, and it is the complement of $\mathcal{Z}$, such that $\mathcal{Z} \cup \mathcal{Z}^c = \rr_+^d$.
Note that when $\bm{z}\in \mathcal{Z}^c$, then $\bm{G}(\bm{z},\bm{x}) \neq \bm{0}$ for any $\bm{x}\in\rr_+^d$, since $\mathcal{Z}^c \times \rr_+^d$ does not intersect $\bm{G}^{-1}(\bm{0})$.
This yields $\bm{G}^{-1}(\bm{0}) = \{ (\bm{z},\bm{x}) : \bm{z} \in \mathcal{Z}, \bm{G}(\bm{z},\bm{x}) = \bm{0} \}$, and using that $\bm{G}(\bm{z},\bm{f}(\bm{z})) = \bm{0}$ for all $\bm{z}\in {\mathscr D}_f$, we find ${\mathscr D}_f \subseteq \mathcal{Z}$, which proves \textit{(ii)}.
However, for $\bm{z}\in \mathcal{Z}$, we may have multiple $\bm{x}\in\rr_+^d$ such that $\bm{G}(\bm{z},\bm{x}) = \bm{0}$, so we investigate this further.

\textit{--- Proof of (iii)}.
We are left with proving uniqueness of $\bm{f}(\cdot)$.
We prove this by considering points $(\bm{z},\bm{x})\in \bm{G}^{-1}(\bm{0})$ and relating them to $\bm{f}(\cdot)$.
From the preimage theorem, we know that $\bm{G}^{-1}(\bm{0})$ is $d$-dimensional, so we need only $d$ parameters to describe this set. 
We can use the implicit function theorem to describe the $\bm{x}$ coordinate of $(\bm{z},\bm{x}) \in \bm{G}^{-1}(\bm{0})$ in terms of an implicit function of $\bm{z}$.
We consider a particular point in this set and then show how the argument extends to other points.

Consider the point $(\bm{z},\bm{x}) = (\bm{1},\bm{1}) \in \bm{G}^{-1}(\bm{0})$ since it satisfies $\bm{G}(\bm{1},\bm{1}) = \bm{0}$, where we use Assumption~\ref{ass: mgfs B exist} to ensure existence of the moment generating functions of $\bm{B}_j$ around this point.
Evaluated at the point $(\bm{1},\bm{1})$, the Jacobian of $\bm{G}$ with respect to $\bm{x}$ is given by
\begin{align}
    \bm{J}_{\bm{G},\bm{x}}(\bm{1},\bm{1})
    = \bm{H}^\top - \bm{I},
\end{align}
which is invertible due to Assumption~\ref{ass: stability condition}.
Then by the implicit function theorem, there exist open sets $V,W\subseteq \rr_+^d$ both containing $\bm{1}$, and a unique continuously differentiable function $\tilde{\bm{f}}:V\to W$ such that $\tilde{\bm{f}}(\bm{1}) = \bm{1}$ and $\bm{G}(\bm{z},\tilde{\bm{f}}(\bm{z})) = \bm{0}$ for all $\bm{z}\in V$.
Note that this implies $V \subseteq \mathcal{Z}$ and that $\tilde{\bm{f}}(\cdot)$ satisfies the fixed-point equation in~\eqref{eq: fixed point equation pgf of S_j (steady state)}.
Moreover, since $\tilde{\bm{f}}(\cdot)$ is unique, we have $\tilde{\bm{f}}(\cdot) = \bm{f}(\cdot)$ on $V$ and $V\subseteq {\mathscr D}_f^\circ$, provided that $\tilde{\bm{f}}(\cdot)$ is increasing in all entries, as by definition $\bm{f}(\cdot)$ is increasing in all entries.

The point $(\bm{1},\bm{1})$ is not particularly special; if we take another point $(\bm{z}_0,\bm{x}_0) \in \bm{G}^{-1}(\bm{0})$, we find that the Jacobian $\bm{J}_{\bm{G},\bm{x}}(\bm{z}_0,\bm{x}_0)$ is invertible provided $(\bm{z}_0,\bm{x}_0) \neq (\hat{\bm{z}},\hat{\bm{x}})$.
We can then apply the implicit function theorem to obtain open sets $\bm{z}_0 \in V_0\subseteq \rr_+^d$, $\bm{x}_0\in W_0\subseteq \rr_+^d$ and a unique map $\tilde{\bm{f}}_0:V_0 \to W_0$ that satisfies $\bm{G}(\bm{z},\tilde{\bm{f}}_0(\bm{z})) = \bm{0}$ for all $\bm{z} \in V_0$, again with $V_0 \subseteq \mathcal{Z}$.
As before, we obtain $\tilde{\bm{f}}(\cdot) = \bm{f}(\cdot)$ on $V_0$ and $V_0 \subseteq {\mathscr D}_f^\circ$, due to uniqueness of $\tilde{\bm{f}}_0$, provided $\tilde{\bm{f}}_0(\cdot)$ is increasing in all entries.
Since we can do this for arbitrary points, we obtain uniqueness of $\bm{f}(\cdot)$ on all of $\mathcal{Z}$, such that $\mathcal{Z} \subseteq {\mathscr D}_f^\circ$.
Finally, for any pair of solutions $(\hat{\bm{z}},\hat{\bm{x}})$ to Eqns.~\eqref{eq: zstar solution multivariate} and~\eqref{eq: system eqn xstar multivariate}, we have by monotonicity of $\bm{f}(\cdot)$ that $\lim_{\bm{z}\nearrow \hat{\bm{z}}} \bm{f}(\bm{z}) = \hat{\bm{x}}$, which yields the characterization $ \mathcal{Z} = {\mathscr D}_f$.
\end{proof}

We remark that taking $d=1$ in Proposition~\ref{prop: characterize f(z) and D_f explicit} yields agreement with the results obtained in \cite[Theorem 3.1.1]{KZ15}, where our condition that the implicit function is increasing, is equivalent to the condition in \cite{KZ15} where they take the minimal solution of the equation $G(z,x) = 0$ for fixed $z<\hat{z}$.
Next, we focus on the limiting cumulant of $\bm{Z}(\cdot)$ given in~\eqref{eq: limit cumulant of Z}.
Note that the moment generating function of $\bm{Z}(t)$ in~\eqref{eq: compound hawkes characterization}, and hence also in \eqref{eq: limit cumulant of Z}, is expressed in terms of the moment generating functions of the random vectors $\bm{U}_1,\dots,\bm{U}_d$.
Denote $\bm{m}_{\bm{U}}(\bm{\theta}) = (m_{\bm{U}_1}(\bm{\theta}),\dots,m_{\bm{U}_d}(\bm{\theta}))^\top$ as the vector of moment generating functions of $\bm{U}_1,\dots,\bm{U}_d$.
We impose the following condition, assumed to hold throughout the paper.

\begin{assumption} \label{ass: mgfs U exist}
Assume that for any $\hat{\bm{z}}$ in~\eqref{eq: zstar solution multivariate}, there exists a vector $\hat{\bm{\theta}}\in \rr^{d^\star}$ such that
\begin{align}
    \bm{m}_{\bm{U}}(\hat{\bm{\theta}}) = \hat{\bm{z}}.
\end{align}
\end{assumption}

Define the function $\Lambda: \rr^{d^\star} \to \rr$ by
\begin{align} \label{eq: definition Lambda in terms of f}
    \Lambda(\bm{\theta}) 
    = \sum_{j=1}^d \overline{\lambda}_j \big(f_j(\bm{m}_{\bm{U}}(\bm{\theta})) - 1\big).
\end{align}
Also define the domain of convergence $ {\mathscr D}_\Lambda:= \{ \bm{\theta} \in \rr^{{d^\star}} : \Lambda(\bm{\theta}) < \infty \}$, denote its interior by ${\mathscr D}^\circ_{\Lambda}$ and denote by $\partial {\mathscr D}_\Lambda$ its boundary.
We now characterize the limiting cumulant of $\bm{Z}(\cdot)$ in~\eqref{eq: limit cumulant of Z}.

\begin{lemma}
\label{lem: limit cumulant of Z(t)}
We have $\bm{0}\in {\mathscr D}^\circ_{\Lambda}$, and for $\bm{\theta} \in \rr^{d^\star}$ such that $\bm{\theta} \leqslant \hat{\bm{\theta}},$ where $\bm{m}_{\bm{U}}(\hat{\bm{\theta}}) = \hat{\bm{z}}$ and with $\hat{\bm{z}}$ the solution to~\eqref{eq: zstar solution multivariate}, we have
\begin{align}
\label{eq: cumulant compound process in terms of clusters}
    \lim_{t\to\infty} \frac{1}{t} \log m_{\bm{Z}(t)}(\bm{\theta})
    =\Lambda(\bm{\theta}).
\end{align}
\end{lemma}

\begin{proof}

We showed that $\bm{1}\in {\mathscr D}_f^\circ$ in the proof of Proposition~\ref{prop: characterize f(z) and D_f explicit}.
We then immediately have by Assumption~\ref{ass: mgfs U exist} that the vector of moment generating functions $\bm{m}_{\bm{U}}(\cdot)$ is defined in a neighborhood of the origin.
Taking $\bm{\theta} = \bm{0}$, we have $\bm{m}_{\bm{U}}(\bm{0})=\bm{1}$, which implies $\bm{0}\in {\mathscr D}^\circ_{\Lambda}$.

We now prove that Eqn.~\eqref{eq: cumulant compound process in terms of clusters} holds.
Combining Eqns.~\eqref{eq: compound hawkes characterization} and \eqref{eq: pgf hawkes characterization}, we obtain 
\begin{align*}
    m_{\bm{Z}(t)}(\bm{\theta})  =
    \prod_{j=1}^d \exp\Big(\overline{\lambda}_j\int_0^t \big(\ee\Big[\prod_{l=1}^d m_{\bm{U}_l}(\bm{\theta})^{S_{l\leftarrow j}(u)}\Big] - 1\big)  \,\mathrm{d}u\Big),
\end{align*}
or, equivalently, 
\begin{align}
\label{eq: proof step: cumulant compound process in terms of clusters}
    \frac{1}{t} \log m_{\bm{Z}(t)}(\bm{\theta}) = \sum_{j=1}^d \overline{\lambda}_j \int_0^1 \big( \ee\Big[\prod_{l=1}^d m_{\bm{U}_l}(\bm{\theta})^{S_{l\leftarrow j}(vt)}\Big] - 1\big)  \,\mathrm{d}v,
\end{align}
performing an elementary change of variables.
We want to take limits $t\to\infty$ in Eqn.~\eqref{eq: proof step: cumulant compound process in terms of clusters}; to that end we first focus on the expectation in the integrand.
By mimicking Part~$(i)$ of the proof of Proposition~\ref{prop: characterize f(z) and D_f explicit}, 
we can apply the monotone convergence theorem on the integrand to find
\begin{align*}
    \lim_{t\to\infty}  \ee\Big[\prod_{l=1}^d m_{\bm{U}_l}(\bm{\theta})^{S_{l\leftarrow j}(vt)}\Big]
    = \ee\Big[\prod_{l=1}^d m_{\bm{U}_l}(\bm{\theta})^{S_{l\leftarrow j}}\Big],
\end{align*}
provided $\bm{\theta} \in {\mathscr D}_\Lambda$.
Now since we took $\bm{\theta} \in {\mathscr D}_\Lambda$, we have by Assumption~\ref{ass: mgfs U exist} that there exists $\hat{\bm{\theta}} \in \rr^{d^\star}$ such that $\bm{\theta} \leqslant \hat{\bm{\theta}}$ and $\bm{m}_{\bm{U}}(\hat{\bm{\theta}}) = \hat{\bm{z}}$.
Using again a similar argument as in Part~$(i)$ in the proof of Proposition~\ref{prop: characterize f(z) and D_f explicit}, distinguishing between indices $k,n\in[d]$ for which $m_{\bm{U}_k}(\bm{\theta}) \leqslant 1$ and $m_{\bm{U}_n}(\bm{\theta})>1$, 
we can apply the dominated convergence theorem to obtain
\begin{align*}
    &\lim_{t\to\infty} \sum_{j=1}^d \overline{\lambda}_j \int_0^1 \big( \ee\Big[\prod_{l=1}^d m_{\bm{U}_l}(\bm{\theta})^{S_{l\leftarrow j}(vt)}\Big] - 1\big)  \,\mathrm{d}v \\
    &=\sum_{j=1}^d \overline{\lambda}_j \int_0^1 \lim_{t\to\infty}\big( \ee\Big[\prod_{l=1}^d m_{\bm{U}_l}(\bm{\theta})^{S_{l\leftarrow j}(vt)}\Big] - 1\big)  \,\mathrm{d}v \\
    &= \sum_{j=1}^d \overline{\lambda}_j \big(f_j(\bm{m}_{\bm{U}}(\bm{\theta})) - 1\big),
\end{align*}
which proves Eqn.~\eqref{eq: cumulant compound process in terms of clusters}.
\end{proof}


\section{Large Deviations}\label{sec:LargeDevs}

In this section, we show that the multivariate compound Hawkes process satisfies a large deviations principle (LDP).
The proof proceeds by establishing the required conditions on the limiting cumulant $\Lambda(\bm{\theta})$---\textit{essential smoothness}, most notably---such that the Gärtner-Ellis theorem (see e.g., \cite[Theorem 2.3.6]{DZ98}) can be invoked.
This section also briefly covers the special case of the LDP of a single component of a multivariate compound Hawkes process.

First recall the definition of an LDP for $\rr^d$-valued random vectors; see \cite[Section 1.2]{DZ98} for more background.
Let $\cB(\rr^d)$ be the Borel $\sigma$-field on $\rr^d$.
Consider a family of random vectors $\{\bm{X}_\epsilon\}_{\epsilon\in\mathbb{R}_{+}}$ taking values in $(\rr^d,  \cB(\rr^d))$.
We say that $\{\bm{X}_\epsilon\}_{\epsilon\in\mathbb{R}_{+}}$ satisfies the LDP with rate function $I(\cdot)$ if $I:\rr^d \to [0,\infty]$ is a lower semicontinuous mapping, and if for every Borel set $A\in\cB(\rr^d)$,
\begin{align}
\label{eq: general def LDP}
    -\inf_{\bm{x}\in A^\circ} I(\bm{x}) \leqslant \liminf_{\epsilon \to 0} \epsilon \log  \pp(\bm{X}_\epsilon\in A) \leqslant \limsup_{\epsilon \to 0} \epsilon \log \pp(\bm{X}_\epsilon \in A) \leqslant - \inf_{\bm{x}\in \overline{A}} I(\bm{x}),
\end{align}
where $A^\circ$ and $\overline{A}$ denote the interior and closure of $A$.
Also, recall that $I(\cdot)$ is lower semicontinuous if, for all $\alpha \geqslant 0$, the level sets $\{\bm{x} \in \rr^d \, : \, I(\bm{x}) \leqslant \alpha\}$ are closed; we call $I(\cdot)$ a \textit{good} rate function if the level sets are compact.

\subsection{LDP of multivariate compound Hawkes processes}\label{ldpM}
In this subsection, we establish that the process $(\bm{Z}(t)/t)_{t\in\rr_+}$ satisfies an LDP on $(\rr^{{d^\star}}, \cB(\rr^{{d^\star}}))$, as stated in the following theorem.

A distinguishing feature of our proof is that, due to the fact that the distribution of $\bm{S}_j$ is not explicitly known, we prove \textit{steepness}---a key step in proving essential smoothness---{\it implicitly}, i.e., through the fixed-point representation~\eqref{eq: fixed point equation pgf of S_j (steady state)} that the probability generating functions $f_j(\cdot)$ satisfy.
In particular, we cannot mimic the proof that was developed in \cite{ST10} for the univariate case, as that proof heavily rests on explicit expressions for the univariate cluster size distribution.

Our steepness proof, as given below, may be somewhat obscured by the involved notation and complex objects needed due to the fact that we work in a multivariate setting. 
To remedy this, we have also included in Appendix~\ref{appendix: steepness} a separate proof for the univariate setting that is based on the same reasoning as the one below, but is considerably more transparent. 

\begin{theorem}\label{thm:LD}
The process $(\bm{Z}(t)/t)_{t\in\rr_+}$ satisfies on $(\rr^{{d^\star}}, \cB(\rr^{{d^\star}}))$ the LDP with good rate function
\begin{align}
    \Lambda^*(\bm{x}) = \sup_{\bm{\theta}\in\rr^{{d^\star}}}
    (\bm{\theta}^\top\bm{x} - \Lambda(\bm{\theta})).
\end{align}
\end{theorem}

\begin{proof}
The proof relies on an application of the Gärtner-Ellis theorem, for which we need to show that the limiting cumulant $\Lambda(\bm{\theta})$ is an essentially smooth, lower semicontinuous function.
For essential smoothness, we need to show that ${\mathscr D}_\Lambda^\circ$ is non-empty and that $\bm{0}\in {\mathscr D}_\Lambda^\circ$, that $\Lambda(\cdot)$ is differentiable on ${\mathscr D}_\Lambda^\circ$, and finally that $\Lambda(\cdot)$ is steep; see \cite[Section 2.3]{DZ98} for further details.

Lemma~\ref{lem: limit cumulant of Z(t)} shows that ${\mathscr D}^\circ_\Lambda$ is non-empty and $\bm{0}\in {\mathscr D}_\Lambda^\circ$.
To show that $\Lambda(\cdot)$ is differentiable on ${\mathscr D}_\Lambda^\circ$, recall from the proof of Proposition~\ref{prop: characterize f(z) and D_f explicit} that $\bm{f}(\cdot)$ is continuously differentiable on ${\mathscr D}_{f}^\circ$, exploiting Assumptions~\ref{ass: stability condition} and~\ref{ass: mgfs B exist}.
Using this property, in combination with the fact that the moment generating functions $m_{\bm{U}_l}(\bm{\theta})$ are differentiable for $\bm{\theta} \in {\mathscr D}_\Lambda^\circ$ by invoking Assumption~\ref{ass: mgfs U exist}, we conclude differentiability of $\Lambda(\cdot)$ on ${\mathscr D}_\Lambda^\circ$.

Next, we prove that $\Lambda(\cdot)$ is steep, i.e., for any $\bm{\bar{\theta}}\in \partial {\mathscr D}_\Lambda^\circ$ and a sequence $\bm{\theta}_n \nearrow \bm{\bar{\theta}}$ as $n\to\infty$, we have that $\lim_{n\to\infty} \norm{\nabla \Lambda(\bm{\theta}_n)}_{\rr^{{d^\star}}} = \infty$.
For any $i\in[{d^\star}]$, we first observe that
\begin{align} \label{eq: partial derivative multivariate limit cumulant}
\begin{split}
    \frac{\partial}{\partial \theta_i} \Lambda(\bm{\theta})
    &= \sum_{j=1}^d \overline{\lambda}_j  \frac{\partial}{\partial \theta_i} f_j(\bm{m}_{\bm{U}}(\bm{\theta}))\\
    &= \sum_{j=1}^d \overline{\lambda}_j \sum_{k=1}^d   \Big(\frac{\partial}{\partial \theta_i} m_{\bm{U}_k}(\bm{\theta})\Big) \ee\Big[ S_{k\leftarrow j} \prod_{l=1}^d m_{\bm{U}_l}(\bm{\theta})^{S_{l\leftarrow j} - \bm{1}_{\{k=l\}}}\Big].
\end{split}
\end{align}
This identity entails that entries of $\nabla\Lambda(\cdot)$ are given in terms of the partial derivatives of the probability generating function of $\bm{S}_j$, for all $j\in[d]$. 

To establish steepness of $\Lambda(\cdot)$, it suffices to show that the partial derivatives of $f_j(\cdot)$ diverge on the boundary of ${\mathscr D}_\Lambda^\circ$.
Recall that the input for the probability generating function ${\bs f}(\cdot)$ in Eqn.~\eqref{eq: definition Lambda in terms of f} is the vector $\bm{m}_{\bm{U}}(\bm{\theta})\in\rr_+^d$.
In the remainder of the proof, we first derive steepness of ${\bs f}(\cdot)$ at a specific $\bm{z}\in\rr_+^d$, after which we consider the setting in which ${\bs f}(\cdot)$ is evaluated in the vector $\bm{m}_{\bm{U}}(\bm{\theta})$.

Define the matrix $\hat{\bm{B}}(\bm{z}) = (\hat{B}_{mj}(\bm{z}))_{m,j\in[d]}$ by
\begin{align}
    \hat{B}_{mj}(\bm{z}) := z_j \ee\Big[B_{mj} c_{mj} \exp\Big(\sum_{i=1}^d B_{ij}c_{ij}(f_i(\bm{z}) - 1)\Big) \Big].
\end{align}
Taking the partial derivative of the fixed-point representation~\eqref{eq: fixed point equation pgf of S_j (steady state)} with respect to $z_k$, for $k\in[d]$, yields
\begin{align} \label{eq: multivariate partial derivative random marks}
\begin{split}
    \frac{\partial f_j(\bm{z})}{\partial z_k} &=
    \ee\Big[\exp\Big(\sum_{m=1}^d B_{mj}c_{mj}(f_m(\bm{z}) - 1)\Big)\Big]\ind_{\{k=j\}} + \sum_{m=1}^d \frac{\partial f_m(\bm{z})}{\partial z_k}\hat{B}_{mj}(\bm{z})
    \\&=\frac{f_j(\bm{z})}{z_j}\ind_{\{k=j\}} +\sum_{m=1}^d \frac{\partial f_m(\bm{z})}{\partial z_k}\hat{B}_{mj}(\bm{z}),
    \end{split}
\end{align}
where the second equality is due to the fixed-point representation~\eqref{eq: fixed point equation pgf of S_j (steady state)} itself.
We can write~\eqref{eq: multivariate partial derivative random marks} compactly in matrix-vector form by considering the Jacobian $\bm{J}_{\bm{f}}$ of $\bm{f}(\cdot)$, which yields
\begin{align} \label{eq: jacobian f wrt z}
    \bm{J}_{\bm{f}}(\bm{z}) = (\bm{I} - \bm{\hat{B}}(\bm{z})^\top)^{-1} \text{diag}(\bm{f}(\bm{z})/\bm{z}),
\end{align}
where the division $\bm{f}(\bm{z})/\bm{z}$ is to be understood componentwise, provided the inverse exists.
We now explore for which values of $\bm{z}$ the inverse appearing in~\eqref{eq: jacobian f wrt z} fails to exist, i.e., when the associated determinant equals 0.
Consider an element $\hat{\bm{z}}$ on the boundary of ${\mathscr D}_f$.
Recall from Eqn.~\eqref{eq: zstar solution multivariate} that this $\hat{\bm{z}}$ is parametrized by some positive vector $\bm{r} \in \rr_+^d$.
Moreover, in this point we have $\bm{f}(\hat{\bm{z}}) = \hat{\bm{x}}$, with $\hat{\bm{x}}$ the solution to~\eqref{eq: system eqn xstar multivariate}, and hence
\begin{align}
    \hat{B}_{mj}(\hat{\bm{z}}) = \hat{z}_j \ee\Big[B_{mj} c_{mj} \exp\Big(\sum_{i=1}^d B_{ij}c_{ij}(\hat{x}_i - 1)\Big) \Big].
\end{align}
Combining this with Eqns.~\eqref{eq: zstar solution multivariate} and~\eqref{eq: system eqn xstar multivariate}, we obtain
\begin{align}
    \hat{\bm{B}}(\hat{\bm{z}}) \cdot \bm{r} = \bm{r} 
    \iff (\bm{I} - \hat{\bm{B}}(\hat{\bm{z}}))\cdot \bm{r} = \bm{0},
\end{align}
implying that $\bm{r}$ is in the kernel of $\bm{I} - \hat{\bm{B}}(\hat{\bm{z}})$.
Since $\bm{r}$ is a positive (non-zero) vector, we obtain that $\bm{I} - \hat{\bm{B}}(\hat{\bm{z}})$ is not invertible and so
\begin{align}
    \det (\bm{I} - \hat{\bm{B}}(\hat{\bm{z}})^\top) = 0.
\end{align}
Then by Eqn.~\eqref{eq: jacobian f wrt z}, we find that the directional derivative into the positive quadrant diverges, i.e., for any $\bm{q}\in\rr_+^d$ and sequence $ \{\bm{z}_n\} \subseteq {\mathscr D}_f^\circ$ such that $\bm{z}_n \nearrow \hat{\bm{z}}$ , we have
\begin{align*}
    \lim_{\bm{z}_n\nearrow \hat{\bm{z}}}\norm{\bm{J}_{\bm{f}}(\bm{z}_n) \cdot \bm{q}}_{\rr^d} = \infty,
\end{align*}
since each element of $\text{diag}(\bm{f}(\hat{\bm{z}})/ \hat{\bm{z}}) =\text{diag}(\hat{\bm{x}}/ \hat{\bm{z}})$ is positive and bounded.
This proves that $\bm{f}(\cdot)$ is steep in each argument.

We now use the above observations to prove steepness of $\Lambda(\cdot)$.
By Assumption~\ref{ass: mgfs U exist}, there exists $\hat{\bm{\theta}}$ on the boundary of ${\mathscr D}_\Lambda$ such that $\bm{m}_{\bm{U}}(\hat{\bm{\theta}}) = \hat{\bm{z}}$.
With the same argument as above, we find $\det (\bm{I} - \hat{\bm{B}}(\bm{m}_{\bm{U}}(\hat{\bm{\theta}}))^\top) = 0$, such that $\bm{I} - \hat{\bm{B}}(\bm{m}_{\bm{U}}(\hat{\bm{\theta}}))^\top$ is not invertible at the boundary of ${\mathscr D}_\Lambda$.
Hence, for any positive vector $\bm{q}\in\rr^d$ and a sequence $\{\bm{\theta}_n\} \subseteq {\mathscr D}_\Lambda^\circ$ such that $\bm{\theta}_n \nearrow \hat{\bm{\theta}}$, we have
\begin{align} \label{eq: jacobian f with mgf U as input diverges}
    \liminf_{\bm{\theta}_n \nearrow \hat{\bm{\theta}}}\norm{\bm{J}_{\bm{f}}(\bm{m}_{\bm{U}}(\bm{\theta}_n)) \cdot \bm{q}}_{\rr^d} = \infty.
\end{align}
If we denote the entries of $\bm{J}_{\bm{f}}$ by $\bm{J}_{\bm{f}}^{(jk)} = \partial f_j/\partial z_k$, then from Eqn.~\eqref{eq: partial derivative multivariate limit cumulant}, we have
\begin{align}\label{steep}
    &\hspace{-5mm}\liminf_{\bm{\theta}_n \nearrow \hat{\bm{\theta}}} \norm{\nabla \Lambda(\bm{\theta}_n)}_{\rr^{{d^\star}}} 
    = \liminf_{\bm{\theta}_n \nearrow \hat{\bm{\theta}}} \Big\lVert \Big( \frac{\partial}{\partial \theta_{1}} \Lambda(\bm{\theta}_n), \dots, 
     \frac{\partial}{\partial \theta_{{d^\star}}} \Lambda(\bm{\theta}_n) \Big)\Big\rVert_{\rr^{{d^\star}}} \notag \\
    &=
    \Big\lVert \Big(\sum_{j=1}^d \overline{\lambda}_j \sum_{k=1}^d \liminf_{\bm{\theta}_n \nearrow \hat{\bm{\theta}}} \frac{\partial}{\partial \theta_{1}} m_{\bm{U}_k}(\bm{\theta}_n)  \ee\Big[ S_{k\leftarrow j} \prod_{l=1}^d m_{\bm{U}_l}(\bm{\theta}_n)^{S_{l\leftarrow j} - \bm{1}_{\{k=l\}}}\Big],  \notag \\
    &\quad\quad \dots, 
    \sum_{j=1}^d \overline{\lambda}_j \sum_{k=1}^d  \liminf_{\bm{\theta}_n \nearrow \hat{\bm{\theta}}} \frac{\partial}{\partial \theta_{{d^\star}}} m_{\bm{U}_k}(\bm{\theta}_n)
     \ee\Big[ S_{k\leftarrow j} \prod_{l=1}^d m_{\bm{U}_l}(\bm{\theta}_n)^{S_{l\leftarrow j} - \bm{1}_{\{k=l\}}}\Big] \Big) \Big\rVert_{\rr^{{d^\star}}} \notag \\
     &\geqslant
    \Big\lVert \Big(\sum_{j=1}^d \overline{\lambda}_j \sum_{k=1}^d  \frac{\partial}{\partial \theta_{1}} m_{\bm{U}_k}(\hat{\bm{\theta}})  \ee\Big[\liminf_{\bm{\theta}_n \nearrow \hat{\bm{\theta}}} S_{k\leftarrow j} \prod_{l=1}^d m_{\bm{U}_l}(\bm{\theta}_n)^{S_{l\leftarrow j} - \bm{1}_{\{k=l\}}}\Big],  \\
    &\quad\quad \dots, 
    \sum_{j=1}^d \overline{\lambda}_j \sum_{k=1}^d  \frac{\partial}{\partial \theta_{{d^\star}}} m_{\bm{U}_k}(\hat{\bm{\theta}})
     \ee\Big[ \liminf_{\bm{\theta}_n \nearrow \hat{\bm{\theta}}} S_{k\leftarrow j} \prod_{l=1}^d m_{\bm{U}_l}(\bm{\theta}_n)^{S_{l\leftarrow j} - \bm{1}_{\{k=l\}}}\Big] \Big) \Big\rVert_{\rr^{{d^\star}}} \notag \\
    &=
    \Big\lVert \Big(\sum_{j=1}^d \overline{\lambda}_j \sum_{k=1}^d \frac{\partial}{\partial \theta_{1}} m_{\bm{U}_k}(\hat{\bm{\theta}}) \bm{J}_{\bm{f}}^{{jk}}(\bm{m}_{\bm{U}}(\hat{\bm{\theta}})), \dots,
    \sum_{j=1}^d \overline{\lambda}_j \sum_{k=1}^d \frac{\partial}{\partial \theta_{{d^\star}}} m_{\bm{U}_k}(\hat{\bm{\theta}})\bm{J}_{\bm{f}}^{{jk}}(\bm{m}_{\bm{U}}(\hat{\bm{\theta}})) \Big) \Big\rVert_{\rr^{{d^\star}}}\notag  \\
    &= \infty,\notag 
\end{align}
where the inequality is an application of Fatou's lemma and the last equality is a consequence of~\eqref{eq: jacobian f with mgf U as input diverges}.

We finally prove lower semicontinuity of $\Lambda(\cdot)$.
Since we consider a metric space $\rr^{{d^\star}}$, it suffices to show lower semicontinuity through sequences.
Consider $\bm{\theta}_n \nearrow \bm{\theta} \in {\mathscr D}_\Lambda^\circ$ and observe that by Fatou's lemma, we have
\begin{align}
    \liminf_{\bm{\theta}_n \nearrow \bm{\theta}} \ee\Big[ \prod_{l=1}^d m_{\bm{U}_l}(\bm{\theta}_n)^{S_{l\leftarrow j}}\Big] 
    \geqslant \ee\Big[ \prod_{l=1}^d   \liminf_{\bm{\theta}_n \nearrow \bm{\theta}} m_{\bm{U}_l}(\bm{\theta}_n)^{S_{l\leftarrow j}}\Big],
\end{align}
for any $j\in[d]$.
Furthermore, it is easily shown that, for any integer $k\in\nn$, another application of Fatou's lemma yields
\begin{align}
    \liminf_{\bm{\theta}_n \nearrow \bm{\theta}} m_{\bm{U}_l}(\bm{\theta}_n)^k
    = \liminf_{\bm{\theta}_n \nearrow \bm{\theta}} \ee\big[\exp(\bm{\theta}_n^\top \bm{U}_l)\big]^k
    \geqslant \ee\big[ \exp(\bm{\theta}^\top \bm{U}_l )\big]^k.
\end{align}
Since the random variables $S_{l\leftarrow j}$ are non-negative, we obtain
\begin{align}
\begin{split}
    \liminf_{\bm{\theta}_n \nearrow \bm{\theta}} \Lambda(\bm{\theta}_n)
    &= \liminf_{\bm{\theta}_n \nearrow \bm{\theta}} \sum_{j=1}^d \overline{\lambda}_j \big(\ee\Big[\prod_{l=1}^d m_{\bm{U}_l}(\bm{\theta}_n)^{S_{l\leftarrow j}}\Big] - 1\big) \\
    &\geqslant \sum_{j=1}^d \overline{\lambda}_j \big(\ee\Big[ \liminf_{\bm{\theta}_n\nearrow \bm{\theta}}\prod_{l=1}^d m_{\bm{U}_l}(\bm{\theta}_n)^{S_{l\leftarrow j}}\Big] - 1\big)
    \geqslant\Lambda(\bm{\theta}).
\end{split}
\end{align}
We have now verified that the limiting cumulant $\Lambda(\cdot)$ satisfies all conditions for the Gärtner-Ellis theorem \cite[Theorem 2.3.6]{DZ98} to apply. 
This concludes the proof of the LDP.
\end{proof}

The consequence of this LDP is that, for any Borel set $A\in\cB(\rr^{{d^\star}})$, we have that the measure $\nu_t: \cB(\rr^{{d^\star}}) \to [0,1]$ defined by $\nu_t(A) := \pp(\bm{Z}(t)/t \in A)$ satisfies
\begin{align}
    -\inf_{\bm{x}\in A^\circ}\Lambda^*(\bm{x}) \leqslant \liminf_{t\to\infty} \frac{1}{t} \log \nu_t(A)\leqslant 
    \limsup_{t\to\infty} \frac{1}{t} \log \nu_t(A) 
    \leqslant
    -\inf_{\bm{x}\in\overline{A}}\Lambda^*(\bm{x}).
\end{align}

\subsection{LDP of a single component}

As a special case of interest, explicitly required later in this paper, we now provide the LDP of a single component.
Hence, throughout this subsection, fix $i\in[{d^\star}]$ and consider the $\rr$-valued component $Z_i(\cdot)$ of the multivariate compound Hawkes process, as defined in Eqn.~\eqref{eq: compound hawkes entry}. 
The associated limiting cumulant is
\begin{align} \label{eq: def marginal limiting cumulant}
    \Lambda_i(\theta) := \lim_{t\to\infty} \frac{1}{t} \log \ee[e^{\theta Z_i(t)}] = \Lambda(0,\dots,\theta,\dots,0),
\end{align}
where the input vector of $\Lambda(\cdot)$ is non-zero on the $i$-th position.
It is noted that whereas $Z_i(t)$ is a one-dimensional object, it is still driven by the multivariate Hawkes process $\bm{N}(\cdot)$.

Compared to the multivariate setting in Section~\ref{ldpM}, we now have to work with the domain ${\mathscr D}_{\Lambda_i} := \{\theta \in \rr \, : \, \Lambda_i(\theta)<\infty\}$, where the argument of the probability generating function of $\bm{S}_j$ is  given by \[\bm{m}_{\bm{U}_{(i)}}(\theta) := (m_{U_{i1}}(\theta),\dots,m_{U_{id}}(\theta))^\top,\] with $m_{U_{ij}}(\theta) = \ee[e^{\theta U_{ij}}]$.
Here, the vector $\bm{U}_{(i)} = (U_{i1},\dots,U_{id})^\top$ is the $i$-th row of the matrix $\bm{U}$, where $i\in[{d^\star}]$.
From Assumption~\ref{ass: mgfs U exist}, we find that for any $\hat{\bm{z}}$ as the solution to~\eqref{eq: zstar solution multivariate}, there exists $\hat{\theta}>0$ such that $\bm{m}_{\bm{U}_{(i)}}(\hat{\theta}) = \hat{\bm{z}}$,
and for $\theta \leqslant \hat{\theta}$, we have by Lemma~\ref{lem: limit cumulant of Z(t)},
\begin{align}
\label{eq: limiting cumulant Z_i characterization}
    \Lambda_i(\theta) =
    \sum_{j=1}^d \overline{\lambda}_j \big(\ee\Big[\prod_{l=1}^d m_{U_{il}}(\theta)^{S_{l\leftarrow j}}\Big] - 1\big).
\end{align}

\begin{corollary} \label{cor: LD marginal}
The process $(Z_i(t)/t)_{t\in\rr_+}$ satisfies on $(\rr,{\mathcal B}(\rr))$ the LDP with good rate function
\begin{align}
    \Lambda_i^*(x) = \sup_{\theta \in \rr}( \theta x - \Lambda_i(\theta)).
\end{align}
\end{corollary}

\section{Ruin Probabilities}\label{sec:RuinProbs}

In this section, we analyze a risk process in which the claims are generated by a multivariate compound Hawkes process.
In particular, using the LDP results for $Z_i(t)/t$ established in the previous section, we characterize the asymptotic behavior of the ruin probabilities of the corresponding risk process.

We assume a constant premium rate $r > 0$ per unit of time and consider, for a given $i\in[{d^\star}]$, the net cumulative claim process (or: risk process)
\begin{align}
    Y_i(t) := Z_i(t) - rt.\label{defY}
\end{align}
Our objective is to find the asymptotics of the associated ruin probability, i.e., the probability that this net cumulative process ever exceeds level $u$, for some $u > 0$.

From the LLN result for $\bm{Z}(t)/t$ given in Eqn.~\eqref{eq: LLN for compound hawkes Z}, we have that
\begin{align}
    \frac{Z_i(t)}{t} \to \ee[\bm{U}_{(i)}](\bm{I} - \bm{H})^{-1}\overline{\bm{\lambda}},
\end{align}
a.s., with $\ee[\bm{U}_{(i)}] = (\ee[U_{i1}],\dots,\ee[U_{id}])$.
To make sure that ruin is rare, we impose throughout the {\it net profit condition}:
\begin{align}
\label{eq: assumption premium rate r}
    r > \ee[\bm{U}_{(i)}](\bm{I} - \bm{H})^{-1}\overline{\bm{\lambda}},
\end{align}
such that the process $Y_i(t)$ drifts towards $-\infty$. 
For some initial capital $u>0$, the time of ruin is defined as
\begin{align*}
    \tau_u := \inf\{t > 0 \, : \, u + rt - Z_i(t) < 0 \} = \inf\{t > 0 \, : \, Y_i(t) > u \},
\end{align*}
and the associated infinite horizon ruin probability is defined as
\begin{align}
    p(u) := \pp(\tau_u < \infty).
\end{align}
We study the behavior of $p(u)$ for $u$ large. 

From Lemma~\ref{lem: limit cumulant of Z(t)}, it immediately follows that the limiting cumulant function of $Y_i(\cdot)$ satisfies
\begin{align}
    \Psi_i(\theta) := \lim_{t\to\infty} \frac{1}{t}\log \ee\big[e^{\theta Y_i(t)}\big] = \Lambda_i(\theta) -r\theta,
\end{align}
with $\Lambda_i(\cdot)$ given in Eqn.~\eqref{eq: limiting cumulant Z_i characterization}.
By \cite[Lemma 2.3.9]{DZ98}, we know that $\Lambda_i(\cdot)$ is a convex function, which implies that $\Psi_i(\cdot)$ is also convex.
We assume throughout the paper that we are in the light-tailed regime, in the sense that there exists $\theta^\star > 0$ such that 
\begin{align} \label{eq: assumption zero of limit cumulant risk process}
    \Psi_i(\theta^\star) = 0.
\end{align}
In the rest of this section, we prove that $p(u)$ decays essentially exponentially as $u$ increases, as made precise in the following theorem.

\begin{theorem}\label{thm:rp}
For fixed $i\in[{d^\star}]$, the ruin probability $p(u)$ associated to the risk process $Y_i(\cdot)$ has logarithmic decay rate $-\theta^\star$, i.e.,
\begin{align} \label{eq: bound ruin probability marginal}
    \lim_{u\to\infty} \frac{1}{u} \log p(u) = -\theta^\star,
\end{align}
where $\theta^\star$ is the unique positive solution of \eqref{eq: assumption zero of limit cumulant risk process}.
\end{theorem}

\begin{proof}
First note that $\theta^\star>0$ is unique by combining the observations $\Psi_i(0) = 0$, $\Psi_i'(0) < 0$ as a consequence of~\eqref{eq: assumption premium rate r}, and the convexity of $\Psi_i(\cdot)$.
The rest of the proof consists of two main steps, which yield a lower bound and an upper bound on the left-hand side of Eqn.~\eqref{eq: bound ruin probability marginal}.

\noindent \textit{--- Lower bound.} The objective is to prove that
\begin{align}
\label{eq: lower bound ruin probability marginal}
    \liminf_{u\to\infty} \frac{1}{u} \log p(u) \geqslant -\theta^\star.
\end{align}
We first observe that, for any $t>0$, we evidently have $p(u) \geqslant p(u,t) := \pp(Y_i(t) > u)$.
This directly implies that, for any $t>0$,
\begin{align*}
    \liminf_{u\to\infty} \frac{1}{u} \log p(u) 
    &\geqslant \lim_{u\to\infty} \frac{1}{u} \log p(u,ut) \\
    &= \liminf_{u\to\infty} \frac{1}{u} \log \pp(Z_i(ut) - rut > u) \\
    &= t \liminf_{u\to\infty} \frac{1}{ut} \log \pp(Z_i(ut)/ut > r + 1/t).
\end{align*}
Then we can use that, by Corollary~\ref{cor: LD marginal}, $\{Z_i(ut)/ut\}_{t\in\rr_+}$ satisfies the LDP with rate function $\Lambda_i^*(\cdot)$. 
We thus obtain
\begin{align*}
    \liminf_{u\to\infty} \frac{1}{u} \log p(u) 
    \geqslant 
    -t \inf_{x>r+1/t} \Lambda_i^*(x).
\end{align*}
Recall the assumption given in Eqn.~\eqref{eq: assumption premium rate r} that $r> \mu_i:=\ee[\bm{U}_{(i)}](\bm{I} - \bm{H})^{-1}\overline{\bm{\lambda}}$.
Further, note that $\Lambda_i^*(\cdot)$ is continuous and non-decreasing on $(\mu_i,\infty)$ by \cite[Lemma 2.7]{GCW04}, and in addition $\Lambda_i^*(\mu_i) = 0$.
Upon combining all these elements, we have that
\begin{align*}
     \liminf_{u\to\infty} \frac{1}{u} \log p(u) 
    \geqslant -t \Lambda_i^*(r+1/t).
\end{align*}

To compute this final expression, observe that since $\Psi_i(0) = 0$, $\Psi_i(\theta^\star) = 0$ and $\Psi_i(\cdot)$ is convex, we have that $\Psi_i'(\theta^\star) = \Lambda_i'(\theta^\star) - r > 0$.
Since $t>0$ was arbitrary, we can in particular take $t =t^\star:= (\Lambda_i'(\theta^\star) - r)^{-1}$.
This implies that
\begin{align*}
    \Lambda_i^*(r+1/t^\star) 
    &= \Lambda_i^*(\Lambda_i'(\theta^\star)) 
    = \sup_{\theta \in \rr} (\theta \Lambda_i'(\theta^\star) - \Lambda_i(\theta)) \\
    &= \theta^\star \Lambda_i'(\theta^\star) - \Lambda_i(\theta^\star) 
    = \theta^\star( \Lambda_i'(\theta^\star) - r)=\frac{\theta^\star}{t^\star};
\end{align*}
the third equality follows by noting that the supremum is attained at $\theta^\star$ since $g(\theta):=\theta\Lambda_i'(\theta^\star) - \Lambda_i(\theta)$ is maximal if $g'(\theta) = 0$, which is when $\theta= \theta^\star$; the fourth equality follows from $\Psi_i(\theta^\star) = 0$, which is equivalent to $\Lambda_i(\theta^\star) = r\theta^\star$; and the last equality follows from the definition of $t^\star$.
Hence, we conclude that 
\begin{align*}
    \liminf_{u\to\infty} \frac{1}{u} \log p(u) 
    \geqslant -t^\star \Lambda_i^*(r + 1/t^\star)
    = -\theta^\star,
\end{align*}
which proves Eqn.~\eqref{eq: lower bound ruin probability marginal}.

\noindent \textit{--- Upper bound.} We next prove the upper bound, i.e.,
\begin{align} \label{eq: upper bound ruin probability marginal}
    \limsup_{u\to\infty} \frac{1}{u} \log p(u) \leqslant -\theta^\star.
\end{align}
To do so, we majorize the ruin probability $p(u)$ by a finite sum and a sum with countably many terms, as follows.
First, observe that for $u> r$, we have
\begin{align*}
    p(u) \leqslant q(u) :=\pp(  \inf_{n\in\nn} Y_i(n) \geqslant u - r),
\end{align*}
where we have used that the net cumulative claim process can go down with maximally an amount $r$ per time unit.
Define, as before, $t^\star = (\Lambda_i'(\theta^\star) - r)^{-1}$.
For any $L > 0$, the union bound gives
\begin{align*}
    q(u) 
    &\leqslant \sum_{n\leqslant (1+L)t^\star u} \pp(Y_i(n) \geqslant u-r) + \sum_{n> (1+L)t^\star u} \pp(Y_i(n)\geqslant u-r) \\
    &\leqslant \sum_{n\leqslant (1+L)t^\star u} \pp(Y_i(n)\geqslant u-r) + \sum_{n> (1+L)t^\star u} \pp(Y_i(n) \geqslant 0),
\end{align*}
where $n$ takes integer values in both sums.
The intuition is that the first sum contains the contribution of paths corresponding to the most likely timescale $t^\star u$ of exceeding $u$.
This means that the first sum is expected to dominate the second sum as $u\to\infty$, which is proven to be correct in the remainder of the proof.

Consider the finite sum, and note that the Chernoff bound yields for any $\theta > 0$ that
\begin{align} \label{eq: intermediate bound proof upper bound}
\begin{split}
    \sum_{n\leqslant (1+L)t^\star u} \pp(Y_i(n)\geqslant u-r)
    &\leqslant \sum_{n\leqslant (1+L)t^\star u} e^{-\theta(u-r)} \ee\big[ e^{\theta Y_i(n)}\big]\\
    &\leqslant (1+L)t^\star u  e^{-\theta(u-r)} \max_{n\in [(1+L)t^\star u]} \ee\big[ e^{\theta Y_i(n)}\big].
\end{split}
\end{align}
In particular, this inequality holds for $\theta = \theta^\star$.
Inserting the familiar expression for the cumulant of $Z_i(t)$, we have for any $t>0$ that
\begin{align*}
    \log \ee\big[e^{\theta^\star Y_i(t)}\big]
    &= \log \ee\big[e^{\theta^\star Z_i(t)}\big] - r\theta^\star t \\
    &= \sum_{j=1}^d \overline{\lambda}_j\int_0^t \ee\bigg[\prod_{l=1}^d m_{U_{il}}(\theta^\star)^{S_{l\leftarrow j}(v)}\bigg] \mathrm{d}v -  t\sum_{j=1}^d \overline{\lambda}_j - r\theta^\star t \\
    &\leqslant t\sum_{j=1}^d \overline{\lambda}_j  \sup_{v\leqslant t} \ee\bigg[\prod_{l=1}^d m_{U_{il}}(\theta^\star)^{S_{l\leftarrow j}(v)}\bigg] - t\sum_{j=1}^d \overline{\lambda}_j - r\theta^\star t \\
    &\leqslant t \left( \sum_{j=1}^d \overline{\lambda}_j \ee\bigg[\prod_{l=1}^d m_{U_{il}}(\theta^\star)^{S_{l\leftarrow j}}\bigg] - \sum_{j=1}^d\overline{\lambda}_j - r\theta^\star\right) = t \Psi_i(\theta^\star)=  0,
\end{align*}
where in the second inequality we used that $S_{l\leftarrow j}(v)$ increases in $v$, and in the final equality we used the definition of $\theta^\star$, as given in Eqn.~\eqref{eq: assumption zero of limit cumulant risk process}.
The next step is to combine this with the upper bound that was found in Eqn.~\eqref{eq: intermediate bound proof upper bound}.
This yields
\begin{align*}
    \lefteqn{\limsup_{u\to\infty} \frac{1}{u} \log \sum_{n\leqslant (1+L)t^\star u} \pp(Y_i(n)\geqslant u-r) }\\&\leqslant\limsup_{u\to\infty} \frac{1}{u}\log\Big( (1+L)t^\star u \max_{n\in [(1+L)t^\star u]} e^{-\theta^\star(u-r)} \ee\big[ e^{\theta^\star Y_i(n)}\big] \Big) \\
    &\leqslant \limsup_{u\to\infty} \frac{1}{u}\log\big((1+L)t^\star u) + \lim_{u\to\infty} \frac{-\theta^\star u + \theta^\star r}{u} + 0= -\theta^\star.
\end{align*}
For the sum with countably many terms, we again apply the Chernoff bound.
For any $\theta > 0$, we thus obtain
\begin{align*}
    \sum_{n> (1+L)t^\star u} \pp(Y_i(n) \geqslant 0) 
    \leqslant \sum_{n> (1+L)t^\star u} \ee\big[e^{\theta Y_i(n)}\big].
\end{align*}
To proceed, we observe that the assumption $\Psi_i(\theta^\star) = 0$ together with $\Psi_i(0) = 0$ implies, by the mean value theorem, that there exists $\theta^\circ >0 $ such that $\Psi_i'(\theta^\circ)=0$.
It requires elementary calculus to verify that
\begin{align*}
    \Psi_i'(0)
    &= \Lambda_i'(0) - r= \sum_{j=1}^d \overline{\lambda}_j \sum_{k=1}^d \ee[U_{ik}]\,\ee[S_{k\leftarrow j}] -r=\ee[\bm{U}_{(i)}]\,\ee[\bm{S}]\,\bm{\overline{\lambda}} -r,
\end{align*}
cf.\ Eqn.~\eqref{eq: LLN for compound hawkes Z}.
We conclude that $\Psi_i'(0) < 0$ from the net profit condition \eqref{eq: assumption premium rate r}.
Combining the above observations with the fact that $\Psi_i(\cdot)$ is convex, it follows that $\Psi_i(\theta^\circ)<0$.
Hence, we can choose $n_0\in\nn$ and $\delta \in (0, | \Psi_i(\theta^\circ)|)$ such that for all $n$ larger than $n_0$, we have that
\begin{align*}
    \frac{1}{n} \log \ee\big[ e^{\theta^\circ Y_i(n)} \big] < \Psi_i(\theta^\circ) + \delta.
\end{align*}
Finally, since $\Psi_i(\theta^\circ) + \delta < 0$, we have that $z:=\exp(\Psi_i(\theta^\circ) + \delta) < 1$, so that we can apply the geometric series to bound
\begin{align*}
    \sum_{n> (1+L)t^\star u} \ee\big[e^{\theta^\circ Y_i(n)}\big]
    \leqslant \sum_{n> (1+L)t^\star u} \exp\big(n(\Psi_i(\theta^\circ) + \delta)\big)
    \leqslant \frac{z^{(1+L)t^\star u}}{1-z}.
\end{align*}
We conclude the proof by combining the two upper bounds.
We take $L$ large enough such that
$z^{(1+L)t^\star}<e^{-\theta^\star},$
for which it is sufficient that
\[L> \frac{\theta^\star}{ (\Phi(\theta^\circ) + \delta)t^\star}.\]
As a consequence,
\begin{align}
   \limsup_{u\to\infty}\frac{1}{u} \log p(u) \leqslant  \limsup_{u\to\infty}\frac{1}{u}\log\left( (1+L)t^\star u \, e^{-\theta^\star(u-r)}  +  \frac{z^{(1+L)t^\star u}}{1-z}\right)=-\theta^\star.
\end{align}
We have thus established \eqref{eq: upper bound ruin probability marginal}.
\end{proof}

\section{Rare Event Simulation}\label{sec:simulation}

In this section, we show how to use importance sampling to efficiently estimate rare event probabilities.
This is accomplished by first exponentially twisting the underlying probability measure~${\mathbb P}$.
In Section~\ref{CoM}, we describe how to identify the model primitives under this new measure, which we throughout refer to as ${\mathbb Q}$.
In the two subsequent subsections, we specifically consider the probability of ruin in component $i\in[d]$ (of which the logarithmic asymptotics have been derived in Theorem~\ref{thm:rp}), and the probability of the multivariate compound Hawkes process attaining rare values (of which the logarithmic asymptotics have been established in Theorem~\ref{thm:LD}). 

\subsection{Identification of the alternative distribution}\label{CoM}
In this subsection, we describe how to construct the exponentially twisted version of the multivariate compound Hawkes process, which we associate with the probability measure ${\mathbb Q}$, without having explicit expressions for the moment and probability generating functions of the original process.
More specifically, we identify a stochastic process of which the limiting cumulant equals, for a vector ${\bs\theta}^\star\in{\mathbb R}^{d^\star}$,
\[\Psi^{\mathbb Q}({\bs \theta}) := \Psi({\bs \theta}+{\bs \theta}^\star)- \Psi({\bs \theta}^\star),\]
with 
$\Psi({\bs \theta}) = \Lambda({\bs\theta}) - {\bs r}^\top {\bs\theta},$
and, by virtue of Lemma~\ref{lem: limit cumulant of Z(t)},
\begin{equation}\label{eq:target}\Lambda({\bs\theta})=\sum_{j=1}^d \overline{\lambda}_j \big(f_j(\bm{m}_{\bm{U}}(\bm{\theta})) - 1\big).\end{equation}

To this end, it is first verified that
\[\Lambda({\bs \theta}+{\bs \theta}^\star)- \Lambda({\bs \theta}^\star) =
\sum_{j=1}^d \overline{\lambda}_j f_j(\bm{m}_{\bm{U}}(\bm{\theta}^\star)) \left(\frac{f_j(\bm{m}_{\bm{U}}(\bm{\theta}+{\bs \theta}^\star ))}{f_j(\bm{m}_{\bm{U}}(\bm{\theta}^\star))} - 1\right).\]
Next, for $j\in[d]$,
\begin{align*}
    \frac{f_j(\bm{m}_{\bm{U}}(\bm{\theta}+{\bs \theta}^\star ))}{f_j(\bm{m}_{\bm{U}}(\bm{\theta}^\star))} &=
\frac{1}{f_j(\bm{m}_{\bm{U}}(\bm{\theta}^\star))}\sum_{{\bs n}\in{\mathbb N}_0^d} 
{{\mathbb P}({\bs S}_j={\bs n})}
\prod_{l=1}^d (m_{{\bs U}_l}({\bs\theta}+{\bs \theta}^\star))^{n_l}\\
&=\sum_{{\bs n}\in{\mathbb N}_0^d} 
\frac{{\mathbb P}({\bs S}_j={\bs n})}{f_j(\bm{m}_{\bm{U}}(\bm{\theta}^\star))}\prod_{l=1}^d (m_{{\bs U}_l}({\bs \theta}^\star))^{n_l}\left(\frac{{m}_{{\bs U}_l}(\bm{\theta}+{\bs\theta}^\star)}{{m}_{{\bs U}_l}(\bm{\theta}^\star)}\right)^{n_l}.
\end{align*}
Now define, for $j\in[d]$,
\[{\mathbb Q}({\bs S}_j={\bs n}) := \frac{{\mathbb P}({\bs S}_j={\bs n})}{f_j(\bm{m}_{\bm{U}}(\bm{\theta}^\star))}\prod_{l=1}^d (m_{{\bs U}_l}({\bs \theta}^\star))^{n_l},\]
which induces a probability distribution (i.e., non-negative and summing to 1) by its very construction. 
Define by $f_j^{\mathbb Q}({\bs z})$ the corresponding probability generating function, which is the counterpart of $f_j({\bs z})$ under ${\mathbb Q}$: for $j\in[d]$,
\begin{equation}f_j^{\mathbb Q}({\bs z}) = 
\sum_{{\bs n}\in{\mathbb N}_0^d} {\mathbb Q}({\bs S}_j={\bs n})\,\prod_{l=1}^d z_l^{n_l}
=\frac{f_j(m_{{\bs U}_1}({\bs \theta}^\star)z_1,\ldots,m_{{\bs U}_d}({\bs \theta}^\star)z_d)}{f_j(\bm{m}_{\bm{U}}(\bm{\theta}^\star))}.\label{fqz}\end{equation}
In addition, define, for $l\in[d]$,
\begin{align}\label{densu}
    {\mathbb Q}(U_{1l}\in{\rm d}x_1,\ldots, U_{d^\star l}\in{\rm d}x_{d^\star}):=
    \frac{{\mathbb P}(U_{1l}\in{\rm d}x_1,\ldots, U_{d^\star l}\in{\rm d}x_{d^\star})}{{m}_{{\bs U}_l}(\bm{\theta}^\star)}
     \prod_{k=1}^{d^\star}{e^{\theta_k^\star x_k}},
\end{align}
which generates a probability distribution (i.e., non-negative and integrating to 1); let
${m}^{\mathbb Q}_{{\bs U}_l}(\bm{\theta})$ be the associated moment generating function, given by
\[{m}^{\mathbb Q}_{{\bs U}_l}(\bm{\theta})= 
\frac{{m}_{{\bs U}_l}(\bm{\theta}+{\bs\theta}^\star)}{{m}_{{\bs U}_l}(\bm{\theta}^\star)}. \]
We finally define the base rates under ${\mathbb Q}$ via
\begin{equation}\label{blq}\overline{\lambda}_j^{\mathbb Q}:= \overline{\lambda}_j f_j(\bm{m}_{\bm{U}}(\bm{\theta}^\star)),\end{equation}
for $j\in [d]$.

Upon combining the objects defined above, it now requires an elementary verification to conclude that
\[\Lambda({\bs \theta}+{\bs \theta}^\star)- \Lambda({\bs \theta}^\star) =\sum_{j=1}^d \overline{\lambda}_j^{\mathbb Q} \left(f_j^{\mathbb Q}(\bm{m}^{\mathbb Q}_{\bm{U}}(\bm{\theta})) - 1\right),\]
as desired; cf.\ \eqref{eq:target}.
This means that we have uniquely characterized the joint distribution of the cluster sizes ${\bs S}_j$ (for $j\in[d]$), the joint distribution of the claim sizes ${\bs U}_l$ (for $l\in[d]$), and the base rates under the alternative measure ${\mathbb Q}$.

The only question left is: How does one sample a cluster size ${\bs S}_j$ under ${\mathbb Q}$? 
More concretely: What is the distribution of the marks $B_{lj}$ under the alternative measure ${\mathbb Q}$, and how should the corresponding decay functions $g_{lj}(\cdot)$ be adapted? 
To this end, we revisit~\eqref{eq: fixed point equation pgf of S_j (steady state)}, which we rewrite to 
\begin{equation}\label{fpeP}f_j({\bs z}) = z_j \,m_{{\bs B}_j}(c_{1j}(f_1({\bs z})-1),\ldots,c_{dj}(f_d({\bs z})-1)),\end{equation}
using the self-evident notation
\[m_{{\bs B}_j}({\bs\theta}) := {\mathbb E} \,\exp\left(\sum_{m=1}^d \theta_m B_{mj}\right).\]
Introduce the compact notation ${\bs y}_{{\bs\theta}^\star}({\bs z}):=(m_{{\bs U}_1}({\bs \theta}^\star)z_1,\ldots,m_{{\bs U}_d}({\bs \theta}^\star)z_d)^\top$. 
Hence, as an immediate consequence of \eqref{fqz}, we obtain
\[f_j^{\mathbb Q}({\bs z}) = \frac{f_j({\bs y}_{{\bs\theta}^\star}({\bs z}))}{f_j({\bs y}_{{\bs\theta}^\star}({\bs 1}))}.\]
Upon combining the two previous displays, we conclude that we can rewrite $f_j^{\mathbb Q}({\bs z})$ as
\begin{align*}f_j^{\mathbb Q}({\bs z}) &=  \frac{m_{{\bs U}_j}({\bs \theta}^\star)z_j\cdot m_{{\bs B}_j}(c_{1j}(f_1({\bs y}_{{\bs\theta}^\star}({\bs z}))-1),\ldots,c_{dj}(f_d({\bs y}_{{\bs\theta}^\star}({\bs z}))-1))}{m_{{\bs U}_j}({\bs \theta}^\star)\cdot m_{{\bs B}_j}(c_{1j}(f_1({\bs y}_{{\bs\theta}^\star}({\bs 1}))-1),\ldots,c_{dj}(f_d({\bs y}_{{\bs\theta}^\star}({\bs 1}))-1))}\\
&=z_j\,\frac{m_{{\bs B}_j}(c_{1j}(f_1({\bs y}_{{\bs\theta}^\star}({\bs z}))-1),\ldots,c_{dj}(f_d({\bs y}_{{\bs\theta}^\star}({\bs z}))-1))}{m_{{\bs B}_j}(c_{1j}(f_1({\bs y}_{{\bs\theta}^\star}({\bs 1}))-1),\ldots,c_{dj}(f_d({\bs y}_{{\bs\theta}^\star}({\bs 1}))-1))}
.\end{align*}
To simplify this further, we write
\begin{equation}\label{cQ}c^{\mathbb Q}_{lj}= c_{lj} f_l({\bs y}_{{\bs\theta}^\star}({\bs 1})),\:\:\:\:\bar c^{\mathbb Q}_{lj}:=c^{\mathbb Q}_{lj}-c_{lj},\end{equation}
such that
\[f_j^{\mathbb Q}({\bs z}) = z_j \frac{\displaystyle m_{{\bs B}_j}\left(c_{1j}^{\mathbb Q}\left(\frac{f_1({\bs y}_{{\bs\theta}^\star}({\bs z}))}{f_1({\bs y}_{{\bs\theta}^\star}({\bs 1}))}-1\right)+\bar c^{\mathbb Q}_{1j},\ldots,c_{dj}^{\mathbb Q}\left(\frac{f_d({\bs y}_{{\bs\theta}^\star}({\bs z}))}{f_d({\bs y}_{{\bs\theta}^\star}({\bs 1}))}-1\right)+\bar c^{\mathbb Q}_{dj}
\right)}{m_{{\bs B}_j}(\bar c^{\mathbb Q}_{1j},\ldots,\bar c^{\mathbb Q}_{dj})}.\]
We now focus on the distribution of the marks under the alternative measure ${\mathbb Q}$. 
Denoting $\bar {\bs c}_j^{\mathbb Q}=(\bar c^{\mathbb Q}_{1j},\ldots,\bar c^{\mathbb Q}_{dj})^\top$, we define
\begin{align}\label{densb}
    {\mathbb Q}(B_{1j}\in{\rm d}x_1,\ldots, B_{dj}\in{\rm d}x_d):=
    \frac{{\mathbb P}(B_{1j}\in{\rm d}x_1,\ldots, B_{dj}\in{\rm d}x_d)}
    {{m}_{{\bs B}_j}(\bar {\bs c}_j^{\mathbb Q})} \prod_{l=1}^d e^{\bar c^{\mathbb Q}_{lj} x_l},
\end{align}
so that
\[{m}^{\mathbb Q}_{{\bs B}_j}(\bm{\theta})= 
\frac{{m}_{{\bs B}_j}(\bm{\theta}+\bar {\bs c}_j^{\mathbb Q})}{{m}_{{\bs B}_j}(\bar {\bs c}_j^{\mathbb Q})}. \]
Combining the above relations, we thus conclude that
\[f_j^{\mathbb Q}({\bs z}) = z_j \,m^{\mathbb Q}_{{\bs B}_j}\!\left(c^{\mathbb Q}_{1j}(f^{\mathbb Q}_1({\bs z})-1),\ldots,c^{\mathbb Q}_{dj}(f^{\mathbb Q}_d({\bs z})-1)\right),\]
which has, appealing to Eqn.~\eqref{fpeP}, the right structure. 
This means that we have identified the distribution of the marks and the decay functions under ${\mathbb Q}.$

The following summarizes the above findings. 
Most importantly, the exponentially twisted version of the multivariate compound Hawkes process is again a multivariate compound Hawkes process, but (evidently) with different model primitives.  
Specifically, the ${\bs\theta}^\star$-twisted version of the multivariate compound Hawkes process can be constructed as follows:
\begin{itemize}
    \item[$\circ$] the base rate $\overline{\lambda}_j$ is replaced by $\overline{\lambda}_j^{\mathbb Q}=\overline{\lambda}_j\, f_j(\bm{m}_{\bm{U}}(\bm{\theta}^\star))$; cf.\  \eqref{blq}.
    \item[$\circ$] the density of ${\bs U}_l$ is replaced by ${\mathbb Q}(U_{1l}\in{\rm d}x_1,\ldots, U_{d^\star l}\in{\rm d}x_{d^\star})$, as given by \eqref{densu};
    \item[$\circ$] the density of ${\bs B}_j$ is replaced by ${\mathbb Q}(B_{1j}\in{\rm d}x_1,\ldots, B_{dj}\in{\rm d}x_d)$, as given by \eqref{densb};
     \item[$\circ$] the decay function $g_{lj}(\cdot)$ is replaced by $g_{lj}^{\mathbb Q}(\cdot):=g_{lj}(\cdot)\,f_l({\bs y}_{{\bs\theta}^\star}({\bs 1}))=g_{lj}(\cdot)\,f_l(\bm{m}_{\bm{U}}(\bm{\theta}^\star))$; cf.\ \eqref{cQ}.
\end{itemize}
This exponentially twisting mechanism generalizes the one identified for the \textit{univariate} compound Hawkes process with \textit{unit} marks, featuring in the statement of \cite[Theorem 2.2]{ST10}.

\subsection{Ruin probabilities}\label{RPr}
In this subsection, we return to the problem of assessing the ruin probability pertaining to the net cumulative claim process $Y_i(t)$, as defined by \eqref{defY}. 
We show that twisting ${\bs Y}(t)$ by ${\bs\theta}^\star=(0,\ldots,0,\theta^\star,0,\ldots,0)^{\rm \top}$, with the $\theta^\star$ corresponding to the $i$-th entry and solving $\Psi_i(\theta^\star)=0$, leads to an estimator that is asymptotically efficient (also sometimes referred to as logarithmically efficient, or asymptotically optimal); in the remainder of this subsection, we refer to the alternative measure induced by this specific twist for fixed and given $i$ by ${\mathbb Q}$. 
For more background on optimality notions of importance sampling procedures, such as asymptotic efficiency, we refer to \cite[Section VI.1]{AG07}.
Our proof is given in Appendix~\ref{app:B}. 
In principle, it follows the same structure as the one given in \cite[Section 4]{ST10}; 
therefore, we 
focus on the main innovations in this general, multivariate setting with random marks.

Recall that $p(u) = {\mathbb P}(\tau_u <\infty)$, with $\tau_u$ the first time that $Y_i(\cdot)$ exceeds level $u$.
First note that $p(u)={\mathbb E}_{\mathbb Q}[ L_{\tau_u}I]$, where $I$ is the indicator function of the event $\{\tau_u<\infty\}$ and $L_{\tau_u}$ is the appropriate likelihood ratio, which quantifies the likelihood of the sampled path under ${\mathbb P}$ relative to ${\mathbb Q}$.
More precisely, $L_{\tau_u}$ is the Radon-Nikodym derivative of the sampled path under the measure ${\mathbb P}$ relative to the measure ${\mathbb Q}$, evaluated at the ruin time $\tau_u$. 
As in \cite[Lemma 4.3]{ST10}, it can be concluded that---essentially due to the fact that we changed the drift of the risk process from a negative value (under ${\mathbb P}$) into a positive value (under ${\mathbb Q}$)---under ${\mathbb Q}$ eventually any positive value is reached by the process $Y_i(\cdot)$. 
Thus, $I\equiv 1$ with ${\mathbb Q}$-probability~1, and hence $p(u)={\mathbb E}_{\mathbb Q}[L_{\tau_u}]$. 

Following the reasoning in \cite{ST10} (i.e., effectively relying on a general result in \cite{JM75}), we can express the likelihood ratio in terms of the various quantities pertaining to the original measure ${\mathbb P}$ and their counterparts under ${\mathbb Q}$. 
Indeed,  the likelihood ratio at time $t$ equals
\begin{align}
\begin{split}
    \frac{\ddiff \pp}{\ddiff \qq}\Big|_{\cF_t} = L_t     &=\exp\left(-\sum_{j=1}^d\int_0^t (\lambda_j(s)-\lambda_j^{\mathbb Q}(s))
    \,{\rm d}s\right)
    \exp\left(\sum_{j=1}^d\int_0^t \log\frac{\lambda_j(s)}{\lambda_j^{\mathbb     Q}(s)}\ddiff N_j(s)\right) \\
    &\quad \times  \exp\left( \sum_{j=1}^d\sum_{r=1}^{N_j(t)} \log\ell_j\big(\bm{B}_{j,r}) \right)
    \prod_{j=1}^d \frac{m_{{\bs U}_j}({\bs\theta}^\star)^{N_j(t)}}{e^{\theta^\star_j Z_j(t)}},
\end{split}
\end{align}
where
\begin{align}\label{lss}
    \lambda_j(s) =\overline{\lambda}_j+\sum_{l=1}^d\sum_{r=1}^{N_l(s)} B_{jl,r}\, g_{jl}(s-T_{l,r}), \quad
    \lambda_j^{\qq}(s) =\overline{\lambda}^{\qq}_j+\sum_{l=1}^d\sum_{r=1}^{N_l(s)} B_{jl,r}\, g_{jl}^{\qq}(s-T_{l,r}),
\end{align}
with $\overline{\lambda}^{\qq}_j$ and $g^{\qq}_{jl}(\cdot)$ as defined in Section~\ref{CoM} and all random objects sampled under ${\mathbb Q}$, and with $\ell_j({\bs x})$ denoting the ratio of the density of the random marks ${\bm{B}}_j$ under ${\mathbb P}$ and its counterpart under ${\mathbb Q}$ evaluated in the argument ${\bs x}$.
It is directly seen, from the construction of the measure ${\mathbb Q}$, that 
\[\lambda_j^{\mathbb Q}(s) = \lambda_j(s) f_j(\bm{m}_{\bm{U}}(\bm{\theta}^\star)) > \lambda_j(s).\]
Note that the relation between $\lambda_j^\qq(s)$ and $\lambda_j(s)$, and the fact that $\bm{\theta}^\star$ is non-zero in the $i$-th entry, allows us to express the likelihood ratio as
\begin{align}\label{eq: likelihood ratio rewritten Rmarks}
\begin{split}
    L_t &= \exp\left(-\sum_{j=1}^d \big(1-f_j(\bm{m}_{\bm{U}}(\bm{\theta}^\star))\big)
    \int_0^t\lambda_j(s)\ddiff s\right) 
    e^{-\theta^\star Z_i(t)} \\
    &\quad \times
    \exp\left( \sum_{j=1}^d\sum_{r=1}^{N_j(t)} \log \ell_j\big(\bm{B}_{j,r}\big)\right)
    \prod_{j=1}^d\left( \frac{m_{\bm{U}_j}(\bm{\theta}^\star)}{f_j(\bm{m}_{\bm{U}}(\bm{\theta}^\star))}\right)^{N_j(t)}.
\end{split}
\end{align}

We now introduce the importance sampling estimator and establish its efficiency.
With $n\in\nn$, we define the importance sampling estimator of $p(u)$ by
\begin{align} \label{eq: def IS estimator}
    p_n(u) := \frac{1}{n} \sum_{m=1}^n L_{\tau_u}^{(m)},
\end{align}
where $L_{\tau_u}^{(m)}$ (for $m=1,\ldots,n$) are independent replications of $L_{\tau_u}$, sampled under ${\mathbb Q}$.
In our context, asymptotic efficiency is to be understood as
\begin{align*}
    \lim_{u\to\infty} \frac{1}{u} \log \sqrt{\mathbb{V}{\rm ar}_{\qq}L_{\tau_u}} \leqslant \lim_{u\to\infty} \frac{1}{u} \log p(u),
\end{align*}
that is, the measure $\mathbb{Q}$ is asymptotically efficient for simulations; see Siegmund's criterion \cite{S76}.

\begin{theorem}\label{aseff}
The importance sampling estimator $p_n(u)$ in \eqref{eq: def IS estimator}, which relies on the alternative measure ${\mathbb Q}$ that corresponds to the exponential twist ${\bs\theta}^\star= (0,\ldots,0,\theta^\star,0,\ldots,0)^{\rm \top}$, is asymptotically efficient. 
\end{theorem}

\begin{remark}{\em 
In the final part of the proof of Theorem~\ref{aseff}, we have in passing derived a Lundberg-type inequality for this non-standard ruin model.
Indeed, we have that the ruin probability $p(u)$, corresponding to the net cumulative claim process $Y_i(\cdot)$, satisfies the upper bound
\begin{align}\label{eq: lundberg bound}
    p(u) \leqslant e^{-\theta^\star u},
\end{align}
uniformly in $u>0$.}
\end{remark}

The immediate consequence of the above theorem, which substantially generalizes \cite[Theorem 4.5]{ST10}, is the following. 
Suppose that we wish to obtain an estimate with a certain {\it precision}, defined as the ratio of the confidence interval's half-width (which is proportional to the standard deviation of the estimate) and the estimate itself. 
Using simulation under the actual measure ${\mathbb P}$, the number of runs required to obtain a given precision is inversely proportional to the probability to be estimated. 
In our specific case, this means that under ${\mathbb P}$, due to Theorem~\ref{thm:rp}, this number grows exponentially in $u$ (roughly like $e^{\theta^\star u}$, that is). 
Under the alternative measure ${\mathbb Q}$, however, Theorem \ref{aseff} entails that the number of runs to achieve this precision grows {\it subexponentially}, thus yielding a substantial variance reduction. 
This means that, despite the fact that the ruin probability decays very rapidly as $u$ grows, the simulation effort required to estimate it grows at a relatively modest pace.

\subsection{Exceedance probabilities}\label{EPr} In this subsection, we consider the estimation of multivariate exceedance probabilities of the type 
\[q_t({\bs a}):={\mathbb P}\left(\frac{Z_1(t)}{t}\geqslant a_1,\ldots,\frac{Z_{d^\star}(t)}{t} \geqslant a_{d^\star}\right),\]
where the set $A:=[a_1,\infty)\times\cdots\times [a_{d^\star},\infty)$
does not contain the vector ${\bs\mu}$, with, as before
\[\mu_i= \ee[\bm{U}_{(i)}](\bm{I} - \bm{H})^{-1}\overline{\bm{\lambda}},\]
the asymptotic value of the process $Z_i(t)/t$.
We consider the regime that $t$ grows large, in which the event of interest becomes increasingly rare by Theorem~\ref{thm:LD}.
We show that the associated importance sampling estimator is asymptotically efficient.

Let $I \equiv I_{\bm{a}}$ be the indicator for the rare event, i.e., set $I_{\bm{a}} := \{ Z_1(t)\geqslant a_1t,\dots, Z_{d^\star}(t)\geqslant a_{d^\star}t\}$ for any given $t>0$.
We define the importance sampling estimator for the probability of this event by
\begin{align} \label{eq: def IS estimator exceedance}
    q_{t,n}(\bm{a}) := \frac{1}{n} \sum_{m=1}^n L_t^{(m)} I^{(m)}_{\bm{a}},
\end{align}
where $L_t^{(m)}$ are independent replications of $L_t$, sampled under $\qq$ with 
a twist parameter depending on $\bm{a}$.
Here, $I^{(m)}_{\bm{a}}$ are the associated indicators.
We state the following theorem, the proof of which is contained in Appendix~\ref{app:C}.
\begin{theorem}\label{aseffq}
The importance sampling estimator $q_{t,n}(\bm{a})$ in \eqref{eq: def IS estimator exceedance}, when using the alternative measure ${\mathbb Q}$ that corresponds to the exponential twist 
\[{\bs\theta}({\bs a}^\star)= \arg\sup_{\bs\theta}({\bs\theta}^\top{\bs a}^\star-\Lambda({\bs\theta})),\] with ${\bs a}^\star :=\arg\inf_{{\bs x}\in A}\Lambda^\star({\bs x})$, is asymptotically efficient. 
\end{theorem}

We finish this subsection by briefly discussing a related, intrinsically more complicated, rare-event probability. 
Observe that $q_t({\bs a})$ corresponds to the {\it intersection} of the events $\{Z_1(t)\geqslant a_1t\},\ldots, \{Z_{d^\star}(t)\geqslant a_{d^\star}t\}.$ 
Consider now, rather than the intersection of events, their {\it union}, e.g.,\ in the case $d^\star=2$,
\[{\mathbb P}\left(\frac{Z_1(t)}{t}\geqslant a_1\:\:\mbox{or}\:\:\frac{Z_2(t)}{t}\geqslant a_2\right).\]
Assume that ${\bs \mu}<{\bs a}$, to make sure that we are dealing with a rare event probability, and let $I$ denote the indicator of the union event. 

Let $A^c$ be the complement of $(-\infty,a_1)\times(-\infty,a_2)$.
Theorem~\ref{thm:LD} asserts that, as $t\to\infty$,
\[\frac{1}{t}\log {\mathbb P}\left(\frac{Z_1(t)}{t}\geqslant a_1\:\:\mbox{or}\:\:\frac{Z_2(t)}{t}\geqslant a_2\right) \to -\inf_{{\bs x}\in A^c}\Lambda^\star({\bs x}).\]
With ${\bs a}^\star :=\arg\inf_{{\bs x}\in A^c}\Lambda^\star({\bs x})$, suppose that we twist by ${\bs\theta}({\bs a}^\star)$. 
If it happens that $a_1^\star<a_1$ and $a_2^\star=a_2$, then, following the reasoning applied above, $\theta_1({\bs a})=0$ and $\theta_2({\bs a})>0$. 
When simulating under this measure, one could however have that $I=1$ due to $Z_1(t)\geqslant a_1t$ (despite the fact that it is more likely to see $Z_2(t)\geqslant a_2t$). 
A similar effect occurs when $a_1^\star=a_1$ and $a_2^\star<a_2$. 
The implication is that in those cases we do not have a bound on the likelihood ratio $L_t$ (as opposed to the case of an intersection of events; see the above analysis for $q_t({\bs a})$). 
There are various ways to deal with this inherent complication; see e.g.,\ the discussions on this issue in \cite{D06, M09}. 
The most straightforward solution is to use the asymptotically efficient algorithm featured in Theorem \ref{aseffq} to {\it separately} estimate the three probabilities
\[{\mathbb P}\left(\frac{Z_1(t)}{t}\geqslant a_1,\frac{Z_2(t)}{t}\geqslant a_2\right),\:\:\:{\mathbb P}\left(\frac{Z_1(t)}{t}\geqslant a_1\right),\:\:\:
{\mathbb P}\left(\frac{Z_2(t)}{t}\geqslant a_2\right),\]
and to add up the resulting estimates.



\section{Examples and Numerical Illustrations}\label{sec:Numerics}

In this section, we provide a set of simulation experiments that illustrate the proposed rare event simulation algorithms and assess the achievable efficiency gains relative to conventional simulation methods.
All simulations have been conducted in {\tt Python}; the computer code is available from the authors upon request.
Throughout this section, we consider the bivariate setting for both the Hawkes process $\bm{N}(\cdot) = (N_1(\cdot),N_2(\cdot))^\top$ and  the compound process $\bm{Z}(\cdot) = (Z_1(\cdot),Z_2(\cdot))^\top$, i.e., we set $d = {d^\star} = 2$.

\subsection{Ruin probability}
We focus in this subsection on the net cumulative claim process corresponding to the first component, i.e., $Y_1(t) = Z_1(t) - rt$.
Our objective is to compute the ruin probability 
\[p(u) = {\mathbb P}(\exists \, t>0: Y_1(t)>0)={\mathbb P}(\tau_u<\infty)=\ee_{\qq}[L_{\tau_u}],\]
where we use the notations of  Sections~\ref{sec:RuinProbs} and~\ref{sec:simulation}.
As before, it is assumed that the net profit condition~\eqref{eq: assumption premium rate r} is in place.
Let $p_n(u)$ denote our importance sampling estimator, see Eqn.~\eqref{eq: def IS estimator}.

We distinguish between the cases in which the marks are deterministic and random.
It is anticipated that, due to the increased variability of the driving Hawkes process, the ruin probabilities will be larger under random marks than under deterministic marks (obviously assuming that the mean mark sizes in the random mark model equal their counterparts in the deterministic mark model).
By applying our efficient simulation approach we can quantify this effect. 

In the case of deterministic marks, the model primitives are assumed to take the following form. 
For $i,j=1,2$, 
\begin{align*}
    g_{ij}(t) \equiv g_i(t) = e^{-\alpha_i t}, 
    \quad B_{ij} = \beta_{ij},
    \quad U_{ij} \sim \text{Exp}(u_{ij}),
\end{align*}
where $\alpha_i ,\beta_{ij},u_{ij}> 0$.
We note that under deterministic marks, the likelihood ratio $L_{\tau_u}$ given in Eqn.~\eqref{eq: likelihood ratio rewritten Rmarks} simplifies considerably.
Also note that in this case $m_{U_{ij}}(\theta) = u_{ij}(u_{ij} - \theta)^{-1}$ for any $\theta<u_{ij}$ and $i,j=1,2$.
The specific parameters used in the simulation experiments are provided in the captions of the figures and tables that follow.

In order to be able to evaluate the likelihood ratio $L_{\tau_u}$, we first calculate the `twist vector' $(\theta^\star,0)$, where $\theta^\star$ is found by solving $\Psi_1(\theta^\star) = 0$.
We then exponentially twist ${\bs Y}(\cdot)$ by $(\theta^\star,0)$, using the change of measure introduced in Section~\ref{CoM}, enabling us to sample the likelihood ratio $L_{\tau_u}$ under the measure $\qq$, after which we can compute the importance sampling estimator $p_n(u)$.
Recall that under the measure $\qq$, we have that the event $\{\tau_u <\infty\}$ happens with probability one, since the twisted process has positive drift and will hit any level $u>0$.
Figure~\ref{fig: convergence ruin prob det marks} illustrates the validity of Theorem~\ref{thm:rp} by showing the convergence of $u^{-1}\log  p_n(u)$ to the logarithmic decay rate $-\theta^\star$ as $u$ grows large; it also provides insight into the speed of convergence for this specific instance.

\begin{figure}
\includegraphics[scale=.47]{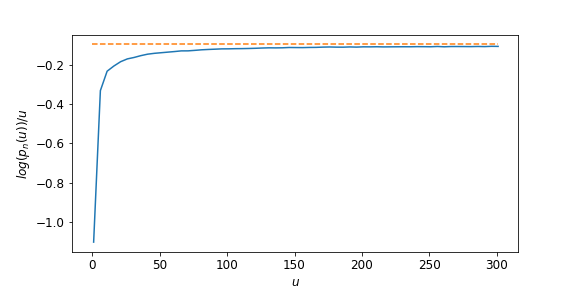}
\caption{\small \textit{Convergence of $u^{-1}\log  p_n(u)$ to the logarithmic decay rate $-\theta^\star$ for the marginal process $Y_1(\cdot)$ in the bivariate model with deterministic marks. 
Chosen parameters are $\overline{\lambda}_1 = \overline{\lambda}_2 = 0.5$, $\alpha_1 = 2$, $\alpha_2 = 1.5$, $\beta_{11} = 0.5$, $\beta_{12} = 0.25$, $\beta_{21} = 0.3$, $\beta_{22} = 0.4$, $\ee U_{11} = 2$, $\ee U_{12} = \ee U_{21} = 2.5$, $\ee U_{22} = 3$ and $r=8$. 
In this setting, solving $\Psi_1(\theta^\star) = 0$ yields $\theta^\star = 0.097$.}}
\label{fig: convergence ruin prob det marks}
\end{figure}

In the case of random marks, we take
\begin{align*}
    g_{ij}(t) \equiv g_i(t) = e^{-\alpha_i t}, 
    \quad B_{ij} \sim \text{Exp}(\gamma_{ij}),
    \quad U_{ij} \sim \text{Exp}(u_{ij}),
\end{align*}
where $\alpha_i, \gamma_{ij}, u_{ij}> 0$. 
To study the effect of the marks being random rather than deterministic, we take ${\mathbb E}B_{ij}=1/\gamma_{ij}=\beta_{ij}$ (recalling that the $\beta_{ij}$ were the deterministic marks that we have used in the first experiment).
As before, we need to evaluate the likelihood ratio $L_{\tau_u}$, for which we first solve the equation $\Psi_1(\theta^\star) = 0$.
This requires solving the fixed-point equation~\eqref{eq: fixed point equation pgf of S_j (steady state)} with random marks, where the i.i.d.\ assumption that was imposed on the marks $B_{ij}$ implies that, for $j=1,2$, we have
\begin{align}
    f_j(\bm{z}) = z_j
    \frac{\gamma_{1j}}{\gamma_{1j} - c_{1}(f_1(\bm{z})-1)}
    \frac{\gamma_{2j}}{\gamma_{2j} - c_{2}(f_2(\bm{z})-1)}.
\end{align}
Then we again exponentially twist ${\bs Y}(\cdot)$ by $(\theta^\star,0)$ to sample under the measure $\qq$ and compute $L_{\tau_u}$.
Figure~\ref{fig: convergence ruin prob rand marks} confirms convergence of $u^{-1}\log p_n(u)$ to $-\theta^\star$.
We note that the decay rate $-\theta^\star$ is larger than in the case with deterministic marks, reflecting that the increased variability due to the random marks leads to a larger ruin probability.

\begin{figure}
\includegraphics[scale=.47]{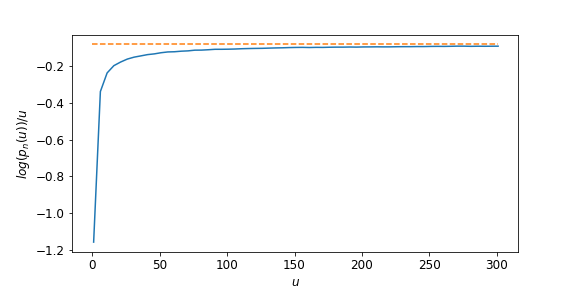}
\caption{\small \textit{Convergence of $u^{-1}\log p_n(u)$ to the logarithmic decay rate $-\theta^\star$ for the marginal process $Y_1(\cdot)$ in the bivariate model with random marks. 
Chosen parameters are $\overline{\lambda}_1 = \overline{\lambda}_2 = 0.5$, $\alpha_1 = 2$, $\alpha_2 = 1.5$, $\ee B_{11} = 0.5$, $\ee B_{12} = 0.25$, $\ee B_{21} = 0.3$, $\ee B_{22} = 0.4$, $\ee U_{11} = 2$, $\ee U_{12} = \ee U_{21} = 2.5$, $\ee U_{22} = 3$ and $r=8$. 
In this setting, solving $\Psi_1(\theta^\star) = 0$ yields $\theta^\star = 0.082$.}}
\label{fig: convergence ruin prob rand marks}
\end{figure}

Next, we study the efficiency of the proposed estimators in terms of the number of runs needed to reach a predefined level of precision.
We continue generating runs until the relative standard error of the importance sampling estimators become smaller than a precision parameter $\epsilon > 0$.
More precisely, given $\epsilon > 0$, we denote the relative standard error (after $n$ runs, that is) and the number of required runs by
\begin{align} \label{eq: relative std error IS estimator}
    \epsilon_n := \frac{\sqrt{v_{n,{\rm IS}}(u)}}{ p_n(u)\sqrt{n}}, \quad \hat{n} := \inf\{n\in{\mathbb N} \, : \, \epsilon_n < \epsilon\},
\end{align}
respectively,
where
\[v_{n,{\rm IS}}(u):=\frac{1}{n} \sum_{m=1}^n \big(L_{\tau_u}^{(m)} -  p_n(u)\big)^2.\]

We denote the number of runs required under deterministic marks by $\hat{n}_{d}$ and under random marks by $\hat{n}_{r}$.
Clearly, $\hat{n}_{d}$ and $\hat{n}_{r}$ vary per experiment; we remedy this by performing the entire procedure multiple times and taking the average.
We also display the associated Lundberg bounds, see Eqn.~\eqref{eq: lundberg bound}.
The twist parameter corresponding to deterministic marks is denoted by $\theta^\star_d$, its counterpart for random marks is denoted by $\theta^\star_r$.
The numbers in Table~\ref{tab: run time bivariate Hawkes} confirm that random marks consistently lead to larger ruin probabilities.
In addition, as expected, the number of runs needed grows in $u$ at a very modest pace (despite the fact that $p(u)$ decays essentially exponentially in $u$).

\begin{table}[t]
{\small 
    \centering
    \begin{tabular}{c| c c c| c c c}
         $u$ & $e^{-\theta^\star_d u}$ & ${p}_{\hat{n}_{d}}(u)$ & $\hat{n}_{d}$ & $e^{-\theta^\star_r u}$ & $p_{\hat{n}_{r}}(u)$ & $\hat{n}_{r}$ \\
          \hline
        1 & 9.07$\,\cdot 10^{-1}$ & 3.15$\,\cdot 10^{-1}$ & 109 & 9.21$\,\cdot 10^{-1}$ & 3.32$\,\cdot 10^{-1}$ & 150 \\ 
        2 & 8.23$\,\cdot 10^{-1}$ & 2.65$\,\cdot 10^{-1}$ & 122 & 8.48$\,\cdot 10^{-1}$ & 2.69$\,\cdot 10^{-1}$ & 173  \\ 
        5 & 6.15$\,\cdot 10^{-1}$ & 1.52$\,\cdot 10^{-2}$ & 128 & 6.62$\,\cdot 10^{-1}$ & 1.75$\,\cdot 10^{-1}$ & 199 \\ 
        10 & 3.79$\,\cdot 10^{-1}$ & 7.89$\,\cdot 10^{-2}$ & 136 & 4.39$\,\cdot 10^{-1}$ & 8.45$\,\cdot 10^{-2}$ & 232 \\
        20 & 1.43$\,\cdot 10^{-1}$ & 2.28$\,\cdot 10^{-2}$ & 142 & 1.92$\,\cdot 10^{-1}$ & 2.68$\,\cdot 10^{-2}$ & 267 \\
        50 & 7.78$\,\cdot 10^{-3}$ & 8.89$\,\cdot 10^{-4}$ & 204 & 1.62$\,\cdot 10^{-2}$ & 1.64$\,\cdot 10^{-3}$  & 386 \\
        100 & 6.05$\,\cdot 10^{-5}$ & 5.83$\,\cdot 10^{-6}$ & 245 & 2.64$\,\cdot 10^{-4}$ & 2.18$\,\cdot 10^{-5}$ & 533 \\
        200 & 3.66$\,\cdot 10^{-9}$ & 3.49$\,\cdot 10^{-10}$ & 302 & 6.95$\,\cdot 10^{-8}$ & 4.70$\,\cdot 10^{-9}$ & 595
    \end{tabular}
    \caption{\label{tab: run time bivariate Hawkes}\small \textit{Lundberg bounds, ruin probabilities, and number of runs needed to reach a precision of $\epsilon = 0.05$, for deterministic (left panel) and random marks (right panel). 
    Chosen parameters are: $\overline{\lambda}_1 = \overline{\lambda}_2 = 0.5$, $\alpha_1 = 2$, $\alpha_2 = 1.5$, $\ee B_{11} = 0.5$, $\ee B_{12} = 0.25$, $\ee B_{21} = 0.3$, $\ee B_{22} = 0.4$, $\ee U_{11} = 2$, $\ee U_{12} = \ee U_{21} = 2.5$, $\ee U_{22} = 3$ and $r=8$. 
    The twist parameters are $\theta_d^\star = 0.097$ and $\theta_r^\star = 0.082$.
    }}
}
\end{table}

In the remainder of this subsection, we consider the setting of random marks.
In the next experiment we assess the computational advantage of the importance sampling estimator (using the measure ${\mathbb Q}$; indicated by subscript IS) when compared to the conventional Monte Carlo estimator (using the measure ${\mathbb P}$; indicated by subscript MC).
Our goal is to compare the time it takes for both estimators to generate a sufficiently precise estimate of $p(u)$.
For the IS estimator, we use Eqn.~\eqref{eq: relative std error IS estimator}.
For the conventional MC estimator based on $n$ runs, denoted by $p_{n,{\rm MC}}(u)$, we have that
\begin{align}
    \epsilon_n = \frac{\sqrt{v_{n,{\rm MC}}(u)}}{p_{n,{\rm MC}}(u)\sqrt{n}} \approx \frac{1}{\sqrt{p_{n,{\rm MC}}(u)\,n}},
\end{align}
since the variance $v_{n,{\rm MC}}(u) := p_{n,{\rm MC}}(u)(1-p_{n,{\rm MC}}(u)) \approx p_{n,{\rm MC}}(u)$ for small probabilities.

In Table~\ref{tab: MC vs IS ruin}, we display the estimates of the ruin probabilities using MC and IS, including the average number of runs needed (in the table denoted by $\hat{n}_{\rm MC}$ and $\hat{n}_{\rm IS}$).
As the absolute duration of each run highly depends on the specific hardware used, the programming language, the number of cores, etc., we decided to work with the {\it speedup ratio}, denoted by $\kappa$, which is the ratio of the simulation time needed under MC (to obtain the desired precision, that is) and its counterpart under IS.

Whereas MC is somewhat more efficient for (very) low values of $u$, IS dominates already for moderate values of $u$.
For $u$ larger than $60$, it turned out to be even infeasible to obtain an MC estimate within a reasonable amount of time, whereas IS estimates can still be efficiently obtained. 
For instance, by extrapolating the results we found for smaller values of $u$, MC would take approximately $18$ hours for $u=70$ in our simulation environment; it would take even $1200$ hours for $u=100$.
In these cases, we {\it estimated} $\kappa$ by extrapolation of the running times under MC (growing effectively exponentially in $u$) and those under IS (growing effectively linearly in $u$); in the table these estimated values are given in italics.
We conclude from the table that the speedup achieved by applying IS can be huge, particularly in the domain that ruin is rare; for $u=300$ the speedup is expected to be as high as $9.00\,\cdot 10^{15}$.


\begin{table}[h!]
    \centering {\small
    \begin{tabular}{c| c c| c c c| c c}
         $u$ & $p_{\hat{n},{\rm MC}}(u)$ & $\hat{n}_{\rm MC}$  & $p_{\hat{n},{\rm IS}}(u)$ & $v_{\hat{n},{\rm IS}}(u)$ & $\hat{n}_{\rm IS}$ & $\kappa$  \\ \hline
        1 & 3.31$\,\cdot 10^{-1}$ & 807  & 3.32$\,\cdot 10^{-1}$ & 4.92$\,\cdot 10^{-2}$ & 150 & 4.62$\,\cdot 10^{-1}$ \\
        2 & 2.71$\,\cdot 10^{-1}$ & 1095 & 2.69$\,\cdot 10^{-1}$ & 3.47$\,\cdot 10^{-2}$ & 173 & 3.60$\,\cdot 10^{-1}$  \\
        3 & 2.39$\,\cdot 10^{-1}$ & 1357 & 2.37$\,\cdot10^{-1}$ & 2.41$\,\cdot 10^{-2}$ & 187 & 3.22$\, \cdot 10^{-1}$ \\
        5 & 1.65$\,\cdot 10^{-1}$ & 2013  &  1.75$\,\cdot 10^{-1}$ & 1.61$\,\cdot 10^{-2}$ & 199 & 3.85$\,\cdot 10^{-1}$  \\
        10 & 8.40$\,\cdot 10^{-2}$ & 4368  & 8.45$\,\cdot 10^{-2}$ & 4.19$\,\cdot 10^{-3}$ & 232 & 5.89$\,\cdot 10^{-1}$ \\
        20 & 2.71$\,\cdot 10^{-2}$ & 14012 & 2.68$\,\cdot 10^{-2}$ & 4.92$\,\cdot 10^{-4}$ & 267 & 1.38$\,\cdot 10^{0}$  \\
        30 & 9.71$\, \cdot 10^{-3}$ & 38639 & 1.02$\,\cdot 10^{-2}$ & 7.29$\,\cdot 10^{-5}$ & 306 & 3.84$\,\cdot 10^{0}$ \\
        40 & 4.08$\,\cdot 10^{-3}$ & 96074 & 4.01$\,\cdot 10^{-3}$ & 1.16$\,\cdot 10^{-5}$ & 355 & 1.12$\,\cdot 10^{1}$ \\
        50 & 1.62$\,\cdot 10^{-3}$ & 263106 & 1.64$\,\cdot 10^{-3}$ & 2.44$\,\cdot 10^{-6}$ & 386 & 3.96$\,\cdot 10^{1}$   \\
        60 & 6.55$\,\cdot 10^{-4}$ & 747083 & 6.46$\,\cdot 10^{-4}$ & 4.36$\,\cdot 10^{-7}$ & 401 & 1.43 $\cdot  10^{2}$ \\
        70 & n/a & n/a & 2.77$\,\cdot 10^{-4}$ & 8.45$\,\cdot 10^{-8}$ & 410 & $\textrm{{\it 4.98}}$ $\cdot \textrm{{\it 10}}^{\,\textrm{{\it 2}}\,}$ \\
        80 & n/a & n/a & 1.12$\,\cdot 10^{-4}$ & 1.64 $\,\cdot10^{-8}$ & 460 & $\textrm{{\it 1.67}}$ $\cdot \textrm{{\it 10}}^{\,\textrm{{\it 3}}\,}$ \\
        100 & n/a & n/a & 2.18$\,\cdot 10^{-5}$ & 5.98$\,\cdot 10^{-10}$ & 533 & $\textrm{{\it 2.04}}$ $\cdot \textrm{{\it 10}}^{\,\textrm{{\it 4}}\,}$  \\ 
        200 & n/a & n/a & 4.70$\,\cdot 10^{-9}$ & 3.51$\,\cdot 10^{-17}$ & 595 & $\textrm{{\it 1.18}}$ $\cdot \textrm{{\it 10}}^{\,\textrm{{\it 10}}\,}$\\
        300 & n/a & n/a & 1.25$\,\cdot 10^{-12}$ & 2.71$\,\cdot 10^{-24}$ & 683 & $\textrm{{\it 9.00}}$ $\cdot \textrm{{\it 10}}^{\,\textrm{{\it 15}}\,}$ \\
    \end{tabular}}
    \caption{\small \textit{Ruin probabilities, number of runs needed to reach a precision of $\epsilon = 0.05$, and speedup ratio $\kappa$, using Monte Carlo (MC) and Importance Sampling (IS). 
    Chosen parameters are as in the caption of Table~\ref{tab: run time bivariate Hawkes}.}}
    \label{tab: MC vs IS ruin}
\end{table}


\subsection{Exceedance probability}

In this subsection, we numerically illustrate the rare event simulation procedure proposed in Section~\ref{EPr}.
We consider the simulation-based computation of the bivariate exceedance probability
\begin{align*}
    q_t(a_1,a_2) = \pp\left(\frac{Z_1(t)}{t}\geqslant a_1, \frac{Z_2(t)}{t}\geqslant a_2 \right),
\end{align*}
where we assume $(a_1,a_2) \not\leqslant (\mu_1,\mu_2) := \lim_{t\to\infty} (Z_1(t)/t, Z_2(t)/t)$ to ensure that we are dealing with an event that becomes increasingly rare as $t\to\infty.$
By Theorem~\ref{thm:LD},
\begin{align} \label{eq: ldp for exceedance bivariate}
    \lim_{t\to\infty}\frac{1}{t} \log q_t(a_1,a_2) = - \inf_{(x_1,x_2)\in A}\Lambda^\star(x_1,x_2),
\end{align}
where $A = [a_1,\infty)\times [a_2,\infty)$.
Denote the minimizer of the RHS of~\eqref{eq: ldp for exceedance bivariate} by $\bm{a}^\star = (a_1^\star, a_2^\star)$, which can be obtained using standard optimization techniques since $\Lambda^\star(\cdot)$ is a convex function.

Consider estimating $q_t(a_1,a_2)$ by importance sampling.
More precisely, with $I_{(a_1,a_2)}$ denoting the indicator of the event we are interested in, and $L_t$ the likelihood ratio given in Eqn.~\eqref{eq: likelihood ratio rewritten Rmarks}, we have
\begin{align}
    q_t(a_1,a_2) = \ee_\qq[L_t I_{(a_1,a_2)}].
\end{align}
Let $q_{t,n}(a_1,a_2)$ be the importance sampling estimator for $q_t(a_1,a_2)$ as defined in Eqn.~\eqref{eq: def IS estimator exceedance}.
To find the twist parameter, we solve the optimization problem
\begin{align*}
    \bm{\theta}(\bm{a}^\star) 
    = (\theta_1(a_1^\star,a_2^\star),\theta_2(a_1^\star,a_2^\star))
    = \arg\sup_{\bs\theta} (\bm{\theta}^\top \bm{a}^\star - \Lambda(\bm{\theta})).
\end{align*}

Figure~\ref{fig: convergence exc prob rand marks} illustrates the behavior of $q_t(a_1,a_2)$ as $t$ grows, converging to $-\inf_{\bm{x}\in A} \Lambda^\star(\bm{x}) = -\Lambda^\star(\bm{a})$, as stated in Eqn.~\eqref{eq: ldp for exceedance bivariate}.
We choose $a_1 = 10 > 3.90 = \mu_1$ and $a_2=12 > 4.76 = \mu_2$, such that $\{Z_1(t) \geqslant a_1 t, Z_2(t) \geqslant a_2 t\}$ is an increasingly rare event as $t$ grows.
For our specific parameters $\bm{\theta}(\bm{a}^\star) > \bm{0}$ (componentwise, that is), corresponding to $\bm{a}^\star = \bm{a} = (a_1,a_2)$.
Intuitively, this means that the most probable way in which the process $(Z_1(t)/t,Z_2(t)/t)_{t\in\rr_+}$  reaches the region $A=[a_1,\infty)\times[a_2,\infty)$ is a straight line from  the origin to $(a_1,a_2)$.

Next, we quantify the computational advantage of the IS estimator over the conventional MC estimator, using the same approach as the one underlying Table~\ref{tab: MC vs IS ruin}.
As before, we run simulations for both methods until the relative standard error $\epsilon_n$ is under the desired level of precision $\epsilon$.
Table~\ref{tab: MC vs IS exceedance} displays the comparison between MC and IS, for different values of $t$. 
As $t$ increases, MC becomes infeasible due to the very steeply increasing number of runs needed as well as the simulation time needed per run.
Already at $t=15$, it would take approximately $69$ hours in our simulation environment.
IS, however, remains feasible, even in the domain of extremely small probabilities.
Note that the number of runs needed for IS does initially not increase in a monotone fashion, which is due to the fact that for small $t$ the process is not yet in the regime where the exceedance event is rare.
We also note the speedup ratio $\kappa$ of the exceedance probabilities increases more steeply (in $t$) than that of the ruin probabilities (in $u$).

\begin{figure}
\includegraphics[scale=.47]{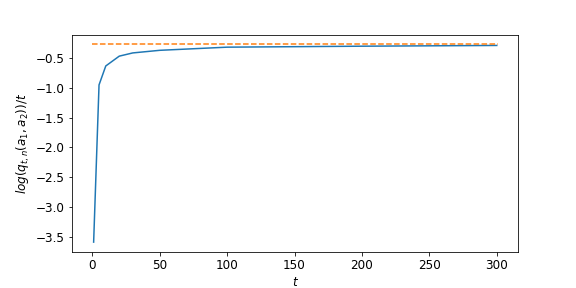}
\caption{\small \textit{Convergence of $t^{-1}\log q_{t,n}(a_1,a_2)$ to the logarithmic decay rate $-\Lambda^\star(\bm{a}^\star)$ for the bivariate process $(Z_1(t)/t,Z_2(t)/t)_{t\in\rr_+}$. 
Chosen parameters are: $a_1 = 10$, $a_2 = 12$, $\overline{\lambda}_1 = \overline{\lambda}_2 = 0.5$, $\alpha_1 = 2$, $\alpha_2 = 1.5$, $\ee B_{11} = 0.5$, $\ee B_{12} = 0.25$, $\ee B_{21} = 0.3$, $\ee B_{22} = 0.4$, $\ee U_{11} = 2$, $\ee U_{12} = \ee U_{21} = 2.5$, $\ee U_{22} = 3$. 
Since $\bm{a}^\star = \bm{a}$, the decay rate is $-\Lambda^\star(\bm{a}) = -0.276$ and the twist parameter is $\bm{\theta}^\star(\bm{a}) = (0.0376, 0.0256)$.}}
\label{fig: convergence exc prob rand marks}
\end{figure}


\begin{table}[h!]
    \centering {\small 
    \begin{tabular}{c| c c| c c c| c c }
         $t$ & $q_{t,\hat{n},{\rm MC}}(\bm{a})$ & $\hat{n}_{\rm MC}$ & $q_{t,\hat{n},{\rm IS}}(\bm{a})$ & $v_{t,\hat{n},{\rm IS}}({\bs a})$ & $\hat{n}_{\rm IS}$ & $\kappa$ \\ \hline
        1  & 2.58$\,\cdot 10^{-2}$ & 16620 & 2.61$\,\cdot 10^{-2}$ & 8.27$\,\cdot 10^{-3}$ & 4879 & 6.56$\,\cdot 10^{-1}$ \\
        2 & 1.83$\,\cdot 10^{-2}$ & 21983 & 1.78$\,\cdot 10^{-2}$ & 3.47$\,\cdot 10^{-3}$ & 4378 & 9.98$\,\cdot 10^{-1}$ \\
        3 & 1.45$\,\cdot 10^{-2}$ & 29076 & 1.42$\,\cdot 10^{-2}$ & 1.92$\,\cdot 10^{-3}$ & 3800 & 1.63$\,\cdot 10^{0}$ \\
        5  & 7.88$\,\cdot 10^{-3}$ & 50869  & 7.95$\,\cdot 10^{-3}$ & 5.67$\,\cdot 10^{-4}$ & 3589 & 3.66$\,\cdot 10^{0}$  \\
        10 & 1.75$\,\cdot 10^{-3}$ & 205934 &  1.69$\,\cdot 10^{-3}$ & 2.73$\,\cdot 10^{-5}$ & 3818 & 3.04$\,\cdot 10^{1}$ \\
        15 & n/a & n/a & 3.64$\,\cdot 10^{-4}$ & 1.40$\,\cdot10^{-6}$ & 4223 & $\textrm{{\it 3.24}}$ $\cdot \textrm{{\it 10}}^{\,\textrm{{\it 2}}\,}$ \\
        20 & n/a & n/a & 7.83$\,\cdot 10^{-5}$ & 7.10$\,\cdot 10^{-8}$ & 4631 & $\textrm{{\it 3.92}}$ $\cdot \textrm{{\it 10}}^{\,\textrm{{\it 3}}\,}$ \\
        25 & n/a & n/a & 1.87$\,\cdot 10^{-5}$ & 4.48$\,\cdot10^{-9}$ & 5123 & $\textrm{{\it 5.04}}$ $\cdot \textrm{{\it 10}}^{\,\textrm{{\it 4}}\,}$ \\
        30 & n/a & n/a & 3.84$\,\cdot 10^{-6}$ & 2.23$\,\cdot 10^{-10}$ & 6043 & $\textrm{{\it 6.25}}$ $\cdot \textrm{{\it 10}}^{\,\textrm{{\it 5}}\,}$ \\
        40 & n/a & n/a & 2.18$\,\cdot 10^{-7}$ & 8.05$\,\cdot 10^{-13}$ & 6749 & $\textrm{{\it 1.28}}$ $\cdot \textrm{{\it 10}}^{\,\textrm{{\it 8}}\,}$ \\
        50 & n/a & n/a & 1.15$\,\cdot 10^{-8}$ & 2.69$\,\cdot 10^{-15}$ & 8209 & $\textrm{{\it 3.43}}$ $\cdot \textrm{{\it 10}}^{\,\textrm{{\it 10}}\,}$  \\
        75 & n/a & n/a & 9.25$\,\cdot 10^{-12}$ & 2.16$\,\cdot 10^{-24}$ & 10106 & $\textrm{{\it 3.93}}$ $\cdot \textrm{{\it 10}}^{\,\textrm{{\it 16}}\,}$  \\
        100& n/a & n/a & 6.31$\,\cdot 10^{-15}$ & 1.44$\,\cdot 10^{-27}$ & 11926 & $\textrm{{\it 5.56}}$ $\cdot \textrm{{\it 10}}^{\,\textrm{{\it 22}}\,}$
    \end{tabular}}
    \caption{\small \textit{Estimation of the exceedance probability $q_t(\bm{a})$, number of runs needed to reach a precision of $\epsilon = 0.05$, and speedup ratio $\kappa$, using Monte Carlo (MC) and Importance Sampling (IS).  
    Chosen parameters are as in the caption of Figure~\ref{fig: convergence exc prob rand marks}}}.
    \label{tab: MC vs IS exceedance}
\end{table}

\section{Concluding Remarks}\label{sec:Con}

This paper has established a large deviations principle for multivariate compound Hawkes processes, with the underlying Hawkes processes admitting  general decay functions and random marks. 
In order to prove the LDP, the main technical hurdle concerned proving that the limiting cumulant is steep.
Our steepness proof is methodologically novel, in that we manage to show that the derivative of the cumulant grows to infinity when approaching the boundary of its domain, but, remarkably, without having an explicit characterization of this domain.
Then the logarithmic asymptotics of the corresponding ruin probability are identified, the proof of which relies  (in the lower bound) on the LDP. 
In addition, the logarithmic asymptotics of the multivariate net cumulative claim attaining a rare value (at a given point in time, that is) are established.
The final contribution of the paper concerns the development of rare event simulation procedures based on importance sampling, proven to be asymptotically efficient and analyzed in numerical experiments.

An interesting topic for future research is to consider other types of deviations for multivariate compound Hawkes processes, such as precise or process-level large deviations.
Furthermore, recalling that the steepness proof does not require knowledge of the boundary of the domain, it could also be explored whether our approach carries over to a broader class of processes with an underlying branching structure.

\appendix

\section{Steepness in the Univariate Case} \label{appendix: steepness}

In this appendix, we prove steepness of the limiting cumulant in the univariate case. 
The appendix has two objectives. 
First, in this single-dimensional setting, elements that look intricate in the proof of Theorem~\ref{thm:LD} now simplify and become significantly more transparent; indeed, this univariate proof helps the reader navigating the proof that we gave for the multivariate case. 
Second, our proof is of a generic nature, in that it does not rely on the fact that in this univariate setting the distribution of the cluster size is explicitly known. 
This distinguishes our approach from the one followed in \cite{ST10}, which explicitly uses that the cluster size has a Borel distribution.
In the multivariate case, 
the explicit distribution of the joint cluster size is unknown, 
thus prohibiting the approach of \cite{ST10}.

Consider the case of $d=d^\star=1$; we leave out all indices.
Let $S$ denote the total number of events in a cluster. 
To guarantee stability of the Hawkes process, we assume that the univariate counterpart of Assumption~\ref{ass: stability condition} is in place, that is,
\begin{align}
    \mu = \ee[B] c = \ee[B] \int_0^\infty g(v)\,\ddiff v \in (0,1).
\end{align}
We know that, under the univariate counterparts of Assumptions~\ref{ass: mgfs B exist}--\ref{ass: mgfs U exist},
\begin{align}
     \lim_{t\to\infty} \frac{1}{t} \log \ee\big[e^{\theta Z(t)}\big] =\overline{\lambda}( \ee\big[m_U(\theta)^{S}\big] - 1)=:\Lambda(\theta),
\end{align}
where $m_U(\theta) = \ee[e^{\theta U}]$ is the moment generating function of $U$ and $\overline{\lambda}$ is the base rate.
In this univariate case, the probability generating function $f(\cdot)$ of the cluster size $S$ satisfies the fixed-point representation
\begin{align} \label{eq: univariate random mark fixed point}
    f(z) = z \ee\big[e^{Bc(f(z) - 1)}\big].
\end{align}

First, we take derivatives on both sides of Eqn.~\eqref{eq: univariate random mark fixed point} to obtain
\begin{align}
    f'(z) &= \ee\big[e^{Bc(f(z) - 1)}\big]+  zf'(z)  \ee\big[Bc \,e^{Bc(f(z) - 1)}\big] \notag,
\end{align}
which, due to~\eqref{eq: univariate random mark fixed point}, 
can be rewritten as
\[f'(z)=\frac{f(z)/z}{1-b(z)},\]
where $b(z):= z\, \ee[Bc \,e^{Bc(f(z) - 1)}].$
The next step is to show that there exists $\hat{z}\geqslant 1$ such that $b(\hat{z})=1$.
The domain ${\mathscr D}_f$ is obtained from Proposition~\ref{prop: characterize f(z) and D_f explicit} when $d=1$ as ${\mathscr D}_f = [0,\hat{z}\,]$, where $\hat{z}> 1$ is given by
\begin{align} \label{eq: zstar solution uni}
    \hat{z} = \ee[ Bc \, e^{Bc(\hat{x}-1)} ]^{-1},
\end{align}
and where $\hat{x} >1 $ solves the equation
\begin{align} \label{eq: xstar solution uni}
    x\ee[Bc \, e^{Bc(x-1)} ] = \ee[e^{Bc(x-1)} ],
\end{align}
see also~\cite[Theorem 3.1.1]{KZ15}.
Since $f'(z)$ is well-defined for $0 < z < \hat{z}$, we consider what happens when we approach the boundary $\hat{z}$.
We compute, using Eqns.~\eqref{eq: zstar solution uni} and~\eqref{eq: xstar solution uni} and the fact $f(\hat{z}) = \hat{x}$, that
\begin{align}
    b(\hat{z}) = \hat{z} \, \ee[Bc \,e^{Bc(\hat{x} - 1)}] = 1.
\end{align}
Hence, we obtain
\begin{align}
    \lim_{z\uparrow \hat{z}} f'(z) = \lim_{z\uparrow \hat{z}} \frac{f(z)/z}{1 - b(z)} = \infty,
\end{align}
also noting that for the numerator $f(\hat{z})/\hat{z} = \hat{x}/\hat{z} > 0$.

Concerning the steepness of $\Lambda(\cdot)$, let there be a $\hat{\theta}> 0$ such that $m_U(\hat{\theta}) = \hat{z}$, which is the univariate counterpart of Assumption~\ref{ass: mgfs U exist}.
We then have that
\begin{align}
    \liminf_{\theta \uparrow \hat{\theta}}  \Lambda'(\theta)
    \geqslant \overline{\lambda} \,\ee[U]\, \ee\big[ S m_U(\hat{\theta})^{S-1}\big] \geqslant  \overline{\lambda}\, \ee[U] \,f'(\hat{z})= \infty,
\end{align}
where the first inequality is due to Fatou's lemma, and the second inequality is due to $m_U'(\theta)\geqslant {\mathbb E}[U]$ for any $\theta\geqslant 0$.

\section{Proof of Theorem~\ref{aseff}}\label{app:B}

\begin{proof} 

In order to eventually prove that our estimator is asymptotically efficient, the main idea is to find an upper bound on $L_{\tau_u}$. 
By definition of the conditional intensities $\lambda_j(s)$ and $\lambda_j^{\mathbb Q}(s)$ in~\eqref{lss}, we have
\begin{align*}
    \exp\left(-\int_0^{\tau_u}\sum_{j=1}^d \big(1-f_j(\bm{m}_{\bm{U}}(\bm{\theta}^\star))\big)
   \lambda_j(s)\ddiff s\right) 
    =&\:\exp\left(-\int_0^{\tau_u}\sum_{j=1}^d\overline\lambda_j(1-f_j(\bm{m}_{\bm{U}}(\bm{\theta}^\star)))\,{\rm d}s\right)
    \\
    &\:\hspace{-4.5cm}\times\exp\left(-\int_0^{\tau_u}\sum_{l=1}^d\sum_{j=1}^d\sum_{r=1}^{N_j(s)} B_{lj,r}g_{lj}(s-T_{j,r})(1-f_l(\bm{m}_{\bm{U}}(\bm{\theta}^\star)))\,{\rm d}s\right),
\end{align*}
where we switched the order of the summations over $j$ and $l$.
Then observe that, recalling that $\theta^\star$ solves the equation $\Psi_i(\theta^\star)=0$,
\[\exp\left(-\int_0^{\tau_u}\sum_{j=1}^d\overline\lambda_j(1-f_j(\bm{m}_{\bm{U}}(\bm{\theta}^\star)))\,{\rm d}s\right) =
\exp\left(-{\tau_u}\sum_{j=1}^d\overline\lambda_j(1-f_j(\bm{m}_{\bm{U}}(\bm{\theta}^\star)))\right)= e^{r\theta^\star \tau_u}.\]
In addition, note that $Z_i(\tau_u)-r\tau_u> u$ due to the definition of $\tau_u$, which implies that
\begin{align*}
    e^{r\theta^\star \tau_u}e^{-\theta^\star Z_i(\tau_u)} \leqslant e^{-\theta^\star u}.
\end{align*}
Also observe that since $c_{lj} = \norm{g_{lj}}_{L^1(\rr_+)}$, we have the bound
\begin{align}\notag
    \exp\Bigg(\int_0^{\tau_u}\sum_{j=1}^d\sum_{l=1}^d\sum_{r=1}^{N_j(s)} B_{lj,r}& g_{lj}(s-T_{j,r})(f_j(\bm{m}_{\bm{U}}(\bm{\theta}^\star))-1){\rm d}s\Bigg)\\&\leqslant
    \exp\left(\sum_{j=1}^d\sum_{l=1}^d \sum_{r=1}^{N_j(\tau_u)} B_{lj,r} c_{lj}(f_l(\bm{m}_{\bm{U}}(\bm{\theta}^\star))-1)\right).
    \label{qblj}
\end{align}
Finally, the contribution to the likelihood ratio $L_{\tau_u}$ due to the random marks is given by
\begin{align} \label{eq: LLhd RMarks}
    \ell_j(\bm{v}) = \exp(-\bm{v}^\top\bar{\bm{c}}^{\qq}) \,m_{\bm{B}_j}(\bar{\bm{c}}^{\qq}),
\end{align}
where $\bar{\bm{c}}^{\qq}$ is the twist parameter for $\bm{B}_j$ as defined in Eqn.~\eqref{cQ}.
As a consequence,
\begin{align} \label{eq: LLhd RMarks written out}
    \exp\left( \sum_{j=1}^d \sum_{r=1}^{N_j(\tau_u)} \log \ell_j\big(\bm{B}_{j,r}\big)\right)
    &=  \exp\left(-\sum_{j=1}^d \sum_{r=1}^{N_j(\tau_u)}\bm{B}_{j,r}^\top\bar{\bm{c}}^{\qq}\right) \prod_{j=1}^d m_{\bm{B}_j}\big(\bar{\bm{c}}^{\qq}\big)^{N_j(\tau_u)},
\end{align}
which implies that the expression given in~\eqref{qblj} cancels against the exponential term in~\eqref{eq: LLhd RMarks written out}, noting that $\bar{\bm{c}}^{\qq} = \big( c_{1j}(f_1(\bm{m}_{\bm{U}}(\bm{\theta}^\star))-1), \dots, c_{dj}(f_d(\bm{m}_{\bm{U}}(\bm{\theta}^\star))-1)\big)^\top$.

Upon combining the above, and after some rewriting, we obtain that 
\begin{align*}
    p(u)\leqslant e^{-\theta^\star u}\, {\mathbb E}_{\mathbb Q} \left[\prod_{j=1}^d
    \exp\left(-\log f_j(\bm{m}_{\bm{U}}(\bm{\theta}^\star))+\log m_{{\bm U}_j}({\bm\theta}^\star) + \log m_{\bm{B}_j}(\bar{\bm{c}}^{\qq})
    \right)^{N_j(\tau_u)}\right].
\end{align*}
As it turns out, the expression in the previous display simplifies considerably, as can be seen as follows. 
By \eqref{eq: fixed point equation pgf of S_j (steady state)}, for any ${\bm z}$,
\[\log f_j({\bm z}) = \log z_j + m_{\bm{B}_{j}}\big( c_{1j}(f_1({\bm z})-1), \dots, c_{dj}(f_d(\bm{z})-1)\big).\]
Now plugging in ${\bm z}= \bm{m}_{\bm{U}}(\bm{\theta}^\star)$, we conclude that the expectation under the new measure ${\mathbb Q}$ fully reduces to unity, again by definition of $\bar{\bm{c}}^{\qq}$.
This means that we have arrived at the upper bound $p(u) = {\mathbb E}_{\mathbb Q} L_{\tau_u}  \leqslant e^{-\theta^\star u}$ and we are now in a position to conclude the statement. 
It follows directly from the observation \[{\mathbb V}{\rm ar}_{\mathbb Q}L_{\tau_u} = {\mathbb E}_{\mathbb Q}L_{\tau_u}^2 - ({\mathbb E}_{\mathbb Q} L_{\tau_u} )^2\leqslant {\mathbb E}_{\mathbb Q} L_{\tau_u}^2 \leqslant e^{-2\theta^\star u},\]
in combination with Theorem~\ref{thm:rp}.
\end{proof}

\section{Proof of Theorem~\ref{aseffq}}\label{app:C}

\begin{proof}
We provide the proof for $d^\star = 1$, followed by a proof by example for $d^\star = 2$, which is easily extended for $d^\star > 2$.
For $d^\star=1$, we first observe that Theorem~\ref{thm:LD} yields that, using that we assumed $a_1>\mu_1$,
\[\lim_{t\to\infty}\frac{1}{t} \log q_t(a_1)   = -\inf_{x\geqslant a_1}\Lambda^\star(x)= - \Lambda^\star(a_1).\] Define
\[\theta(a_1):= \arg\sup_\theta (\theta a_1 -\Lambda(\theta));\]
it is straightforward to verify that $\theta(a_1)$ is positive for $a_1>\mu_1$. 
Letting $I$ be the indicator function of the event $\{Z_1(t)\geqslant a_1t\}$, ${\mathbb Q}$ the probability measure corresponding to exponentially twisting the original measure by $\theta(a_1)$, and $L_t$ the appropriate likelihood, we now have that
\[q_t(a_1) = {\mathbb E}_{\mathbb Q}[L_tI].\]
As an aside, observe that in this setting, unlike the one discussed in Section~\ref{RPr}, we do not have that $I=1$ almost surely under ${\mathbb Q}$. 
This is an immediate consequence of the fact that we constructed ${\mathbb Q}$ such that $\lim_{t\to\infty}{\mathbb E}_{\mathbb Q}Z_1(t)/t = a_1$, so that the central limit theorem implies that in roughly half of the runs, we have that $I=1$.
The likelihood ratio can be evaluated by mimicking the calculations in Section~\ref{RPr}. 
We thus obtain, leaving out indices in this single-dimensional case,
\begin{align*}
    L_t=&\: \exp\left(-\big(1-f(m_U(\theta(a_1)))\big)\int_0^t \lambda(s) \,{\rm d}s\right)e^{-\theta(a_1)Z_1(t)}\\
    & \times \exp\left(\sum_{r=1}^{N(t)} \log\ell(B_r)\right)
    \left(\frac{m_U(\theta(a_1))}{f(m_U(\theta(a_1)))}\right)^{N(t)}.
\end{align*}
Using that $\Lambda(\theta) =\overline\lambda(f(m_U(\theta))-1)$ and 
$m_B(c\,(f(z)-1)) - \log f(z) +\log z = 0$ (the latter identity being a consequence of \eqref{eq: fixed point equation pgf of S_j (steady state)}),
and applying essentially the same majorizations as the ones used in Section~\ref{RPr}, we readily obtain that
\[q_t(a_{1}) ={\mathbb E}_{\mathbb Q} [L_tI] \leqslant e^{\Lambda(\theta(a_1))\,t} \,e^{-\theta(a_1)\,Z_1(t)}I.\]
The event $\{I=1\}$ is equivalent to $\{Z_1(t)\geqslant a_{1}t\}$, so that
\[ q_t(a_{1}) \leqslant e^{\Lambda(\theta(a_1))\,t} \,e^{-\theta(a_1)\,a_{1}t} = e^{-\Lambda^\star(a_1)\,t}.\]
We have thus obtained that
\[q_t(a_1) \leqslant e^{-\Lambda^\star(a_1)\,t},\]
which can be seen as a variant of the classical Chernoff bound.
The asymptotic efficiency for $d^\star=1$ now follows directly. 
To this end, first note that using the very same reasoning we also find that
\[{\mathbb E}_{\mathbb Q}[L_t^2I]\leqslant e^{-2\Lambda^\star(a_1)\,t},\]
so that also ${\mathbb V}{\rm ar}_{\mathbb Q}[L_tI]\leqslant e^{-2\Lambda^\star(a_1)\,t}$. 
Combining this with Theorem~\ref{thm:LD}, we conclude that in this single-dimensional case, we have asymptotic efficiency under the measure ${\mathbb Q}$ defined above.

We now move to the case $d^\star=2$, for which Theorem~\ref{thm:LD} gives
\[\lim_{t\to\infty}\frac{1}{t} \log q_t({\bs a}) =-\inf_{(x_1,x_2)\in A}\Lambda^\star({\bs x}).\]
Due to the convexity of the contour lines of $\Lambda^\star({\bs a})$, with ${\bs a}^\star$ the optimizing ${\bs a}\in A$, three situations can occur: (i)~${\bs a}^\star={\bs a}$, (ii)~$a_1^\star=a_1$ and $a_2^\star > a_2$, and (iii)~$a_1^\star>a_1$ and $a_2^\star=a_2$. 
As, by symmetry, cases (ii) and (iii) are conceptually the same and can therefore be treated identically, we restrict ourselves to discussing cases (i) and (ii) only. 
Let, as before, ${\bs \theta}({\bs a}^\star)$ be the optimizing argument in the definition of $\Lambda^\star({\bs a})$.

In case (i), using standard properties of the Legendre transform, we have that
\[\theta_1({\bs a}^\star) = \frac{\partial}{\partial a_1} \Lambda^\star({\bs a}^\star) >0,\:\:\:
\theta_2({\bs a}^\star) = \frac{\partial}{\partial a_2} \Lambda^\star({\bs a}^\star) >0.\]
We let ${\mathbb Q}$ correspond to the ${\bs\theta}({\bs a}^\star)$-twisted version of the original probability measure. 
Going through the same steps as in the case $d^\star=1$, we obtain that
\[q_t({\bs a}) ={\mathbb E}_{\mathbb Q} [L_tI]\leqslant e^{\Lambda({\bs\theta}({\bs a}^\star))\,t}\,e^{-\theta_1({\bs a}^\star)\,Z_1(t)-\theta_2({\bs a}^\star)\,Z_2(t)}.\]
Then note that the right-hand side of the expression in the previous display is, on the set $\{I=1\}=\{Z_1(t)\geqslant a_1^\star \,t,Z_2(t)\geqslant a_2^\star \,t\}$, bounded from above by
\[e^{\Lambda({\bs\theta}({\bs a}^\star))\,t}\,e^{-\theta_1({\bs a}^\star)\,a_1^\star t-\theta_2({\bs a}^\star)\,a_2^\star t}=e^{-\Lambda^\star({\bs a}^\star)\,t}.\]
This implies $q_t({\bs a})\leqslant e^{-\Lambda^\star({\bs a}^\star)\,t}$, but in addition that ${\mathbb V}{\rm ar}_{\mathbb Q}[L_tI]\leqslant e^{-2\Lambda^\star({\bs a}^\star)\,t}$.
We conclude, using the same reasoning as before, that in this case twisting by ${\bs\theta}({\bs a}^\star)$ yields asymptotic efficiency. 

Case (ii) works similarly. 
Observe that now (using that the line $x=a_1$ is a tangent of the contour lines of the Legendre transform)
\[\theta_1({\bs a}^\star) = \frac{\partial}{\partial a_1} \Lambda^\star({\bs a}^\star) >0,\:\:\:
\theta_2({\bs a}^\star) = \frac{\partial}{\partial a_2} \Lambda^\star({\bs a}^\star) =0.\]
The intuition is that in this case, if $Z_1(t)\geqslant a_1 t$, then with high probability also $Z_2(t)\geqslant a_2t$, as reflected by the fact that 
\[\lim_{t\to\infty}\frac{1}{t}\log{\mathbb P}\left(\frac{Z_1(t)}{t}\geqslant a_1,\frac{Z_2(t)}{t} \geqslant a_2\right)=\lim_{t\to\infty}\frac{1}{t}\log
{\mathbb P}\left(\frac{Z_1(t)}{t}\geqslant a_1\right).\]
Let ${\mathbb Q}$ be the ${\bs\theta}({\bs a}^\star)$-twisted version of the original probability measure.  
After straightforward algebra, we now obtain that
\[q_t({\bs a}) ={\mathbb E}_{\mathbb Q}L_tI \leqslant e^{\Lambda({\bs\theta}({\bs a}^\star))\,t}\,e^{-\theta_1({\bs a}^\star)\,Z_1(t)}=
e^{-\Lambda^\star({\bs a}^\star)\,t},\]
and ${\mathbb V}{\rm ar}_{\mathbb Q}[L_tI]\leqslant e^{-2\Lambda^\star({\bs a}^\star)\,t}$. 
Hence, also in this case twisting by ${\bs\theta}({\bs a}^\star)$ yields asymptotic efficiency.

It can be seen in a direct manner that the same procedure (i.e., working with a twist ${\bs\theta}({\bs a}^\star)$ with non-negative entries) 
can be followed for any $d^\star$ larger than $2$. 
We have thus established the stated result.
\end{proof}

\end{document}